\documentclass[9pt]{article}
\usepackage{epsfig}
\usepackage{float}
\usepackage{amssymb,amsmath}

\newcommand{\be}{\begin{equation}}
\newcommand{\ee}{\end{equation}}

\newtheorem{theorem}{Theorem}[section]
\newtheorem{definition}[theorem]{Definition}

\newtheorem{lemma}[theorem]{Lemma}

\newenvironment{proof}{
\noindent {\it Proof.}}{\hfill$\Box$ }

\begin{document}

\title{Eigenvectors of graph Laplacians: a landscape}

\author{ Jean-Guy CAPUTO and Arnaud KNIPPEL }

\maketitle

{\normalsize \noindent Laboratoire de Math\'ematiques, INSA de Rouen Normandie,\\
Normandie Universit\'e \\
76801 Saint-Etienne du Rouvray, France\\
E-mail: caputo@insa-rouen.fr, arnaud.knippel@insa-rouen.fr \\
}

\date{\ }


\begin{abstract}

We review the properties of eigenvectors for the graph Laplacian matrix,
aiming at predicting a specific eigenvalue/vector from the geometry of the graph.
After considering classical graphs for which the spectrum is known, 
we focus on eigenvectors that have zero components and 
extend the pioneering results of Merris (1998) on graph transformations
that preserve a given eigenvalue $\lambda$ or shift it in a simple way.
These transformations enable us to obtain eigenvalues/vectors 
combinatorially instead of numerically; in particular we show 
that graphs having eigenvalues $\lambda= 1,2,\dots,6$ up to six vertices 
can be obtained from a short list of graphs. 
For the converse problem of a $\lambda$ subgraph $G$ of a $\lambda$ graph $G"$,
we prove results and conjecture that $G$ and $G"$ 
are connected by two of the simple transformations described above.

\end{abstract}

\section{Introduction}

The graph Laplacian is an important operator for both theoretical
reasons and applications \cite{crs01}. As its continuous counterpart, it arises 
naturally from conservation laws and has many applications in physics and engineering.
The graph Laplacian has real eigenvalues and eigenvectors can be chosen orthogonal.
This gives rise to a Fourier like description of evolution problems
on graphs; an example is the graph wave equation, a natural model for weak 
miscible flows on a network, see the articles \cite{maas}, \cite{cks13}. 
This simple formalism proved very useful for modeling the electrical grid \cite{ckr19}
or describing an epidemic on a geographical network \cite{ckmk19}. 
Finally, a different application of graph Laplacians is 
spectral clustering in data science, see the review \cite{vonluxburg}.

Almost sixty years ago, Mark Kac \cite{kac} asked the question :
can one Hear the Shape of a Drum? Otherwise said, does the 
spectrum of the Laplacian characterize the graph completely ?
We know now that there are isospectral graphs so that there is
no unique characterization. However, one can ask a simpler question:  
can one predict eigenvalues or eigenvectors from the geometry of the graph? 
From the literature, this seems very difficult, most of the results
are inequalities, see for example the beautiful review by 
Mohar \cite{mohar92} and the extensive monograph \cite{bls07}.

Many of the results shown by Mohar \cite{mohar92} are inequalities
on $\lambda_2$, the first non zero eigenvalue. This eigenvalue is related
to the important maximum cut problem in graph theory and also others.
Mohar \cite{mohar92} also gives some inequalities on $\lambda_n$,
the maximum eigenvalue, in terms of the maximum of the sum of two degrees.
Another important inequality concerns the interlacing of the spectra of
two graphs with same vertices, differing only by an edge.
However, little is known about the bulk of the spectrum, i.e. the
eigenvalues between $\lambda_2$ and $\lambda_n$.
A very important step in that direction was 
Merris's pioneering article \cite{merris} where he introduced
"Laplacian eigenvector principles" that allow to predict how the
spectrum of a graph is affected by contracting, adding or deleting edges 
and/or of coalescing vertices. Also, Das \cite{das04} showed 
that connecting an additional vertex to all vertices of a graph increases
all eigenvalues (except 0) by one.

Following these studies, in \cite{ckk18a} we characterized graphs
which possess eigenvectors of components $\pm 1$ (bivalent) and 
$0,\pm 1$ (trivalent).
This is novel because we give exact results, not inequalities.
Here, we continue on this direction and focus on eigenvectors that have
some zero coordinates, we term these soft nodes;  
such soft nodes are important because there, no action can 
be effected on the associated mechanical system \cite{cks13}. 
In this article, we use the important properties of graphs with
soft nodes, we call these soft-graphs, to highlight 
eigenvalues/eigenvectors that can be obtained
combinatorially (instead of numerically).
We first show that eigenvalues of graph Laplacians with
weights one are integers or irrationals. Then we present well known
classical graphs whose spectrum is known exactly. 
We describe five graph transformations that preserve
a given eigenvalue and two that shift the eigenvalue in a simple way.
Among the transformations that preserve an eigenvalue, 
the link was explicitly introduced in the
remarkable article by Merris ({\it link principle})
\cite{merris}. The articulation and the soldering were contained 
in the same paper and we choose to present 
elementary versions of these transformations.
We find two new transformations that preserve an eigenvalue: the 
regular expansion and the replacement of 
a coupling by a square. We also present transformations that
shift an eigenvalue in a predictable way: insertion of a soft node,
addition of a soft node, insertion of a matching. The 
first is new, the second and third were found by Das \cite{das04} 
and Merris \cite{merris} respectively.

In the last part of the article we enumerate all the small
graphs up to six vertices that have a given eigenvalue $\lambda$
and explain the relations between them using the transformations
discussed previously. It is remarkable that these
graphs can all be obtained from a short list of graphs.
However the question is open for bigger graphs. 
Using the transformations mentioned above, $\lambda$ soft graphs can
be made arbitrarily large. The converse problem of a $\lambda$ 
subgraph $G$ of a $\lambda$ graph $G"$ is considered. We show 
that the matrix coupling the two Laplacians $L(G)$ and $L(G')$,
where $G'= G"-G$, is a graph Laplacian.
If the remainder graph $G'$ is $\lambda$, then it is formed using the
articulation or link transformation. It is possible that the 
remainder graph $G'$ is not $\lambda$ as long as it 
shares an eigenvector with $G$.
Then the two may be related by adding one or several soft nodes 
to $G'$. Finally, an
argument shows that if $G'$ is not $\lambda$ and does not share
an eigenvector with $G$, the problem has no solution.
We finish the article by examining the $\lambda$ soft graphs for
$\lambda=1,2,\dots,6$ and insist on minimal $\lambda$ soft graphs
as generators of these families, using the transformations above. \\
The article is organized as follows. Section 2 introduces the main
definitions. In section 3 we consider special graphs (chains, cycles, cliques,
bipartite graphs) whose Laplacian spectrum is
well known.  The graph transformations preserving an eigenvalue are
presented in section 4. Section 5 introduces graph transformations 
which shift eigenvalues. Finally section 6 introduces $\lambda$ soft graphs,
discusses $\lambda$ sub-graphs
and presents a classification of graphs up to six vertices. 

\section{The graph Laplacian : notation, definitions and properties}
We consider a graph ${G}({V},{E})$ with a vertex set ${V}$ of cardinality 
$n$ and edge set ${E}$ of cardinal $m$ where $n,m$ are finite.
The graph is assumed connected with no loops and no multiple edges.
The graph Laplacian matrix \cite{bls07} is the $(n,n)$ matrix $L(G)$ 
or $L$ such that
\be \label{laplacian} L_{ij}=-1 ~{\rm if~edge~i~j~exists}, 0 ~{\rm otherwise},
~~~~L_{ii}=m_i,~ 
{\rm degree~ of~ i} ,\ee
where the degree of $i$ is the number of edges connected to vertex $i$.

The matrix $L$ is symmetric so that it has real eigenvalues and 
we can always find a basis of orthogonal eigenvectors. 
Specifically we arrange the eigenvalues $\lambda_i$ as
\be\label{eigenvalues}
\lambda_1 = 0 \le \lambda_2 \le \dots \le \lambda_n.\ee
We label the associated eigenvectors $v^1,v^2,\dots,v^n$.

We have the following properties 
\begin{itemize}
\item $v^1={\mathbf{1}}$ the vector whose all components are $1$. 
\item Let $v^i_k$ be the $k$ component of an eigenvector $v^i,~~i>1$.
An immediate consequence of the $v^i$ being orthogonal to $v^1$ is
$\sum_k v^i_k =0$. 
\end{itemize}
A number of the results we present hold when $L_{ij} \neq -1$ 
and $L_{ii} = \sum_{j \sim i}L_{ij}$ , this is the generalized
Laplacian.
We will indicate which as we present them.

{\bf Regular graphs}\\
The graph Laplacian can be written as
$$L = D - A$$
where $A$ is the adjacency matrix and $D$ is the diagonal matrix of
the degrees. \\
We recall the definition of a regular graph.
\begin{definition}[Regular graph]
A graph is $d$-regular if every vertex has the same degree $d$.
\end{definition}
For regular graphs $D = d {\rm Id}_n$, where ${\rm Id}_n$ is the identity
matrix of order $n$. For these graphs, all the properties obtained
for $L$ in the present article carry over to $A$.

We will use the following definitions.
\begin{definition}[{\rm Soft node }]
\label{def1}
A vertex $s$ of a graph is a soft node for an eigenvalue $\lambda$
of the graph Laplacian if there exists an eigenvector $x$ for
this eigenvalue such that $x_s=0$.
\end{definition}
An important result due to Merris \cite{merris} is
\begin{theorem}
Let $G$ be a graph with $n$ vertices. If $0 \neq \lambda < n$
is an eigenvalue of $L(G)$ then any eigenvector affording $\lambda$
has component $0$ on every vertex of degree $n-1$.
\end{theorem}

\begin{definition}[$k$-partite graph]
A $k$-partite graph is a graph whose vertices can be partitioned into $k$ different independent sets
so that no two vertices within the same set are adjacent.
\end{definition}

\begin{definition}[cycle]
A {cycle} is a connected graph where all vertices have degree 2. 
\end{definition}
\begin{definition}[chain]
A {chain} is a connected graph where two vertices have degree 1 and
the other vertices have degree 2.
\end{definition}
\begin{definition}[clique]
A {clique} or complete graph $K_n$ is a simple graph where
every two vertices are connected. 
\end{definition}

In the article we sometimes call configuration a vertex valued graph
where the values correspond to an eigenvector of the graph Laplacian.

\subsection{Eigenvalues are integers or irrationals}

We have the following result
\begin{theorem}
If the eigenvalue $\lambda$ is an integer, then there exist 
integer eigenvectors. 
\end{theorem}
To see this consider the linear system
$$ (L - \lambda I) X = 0 . $$
It can be solved using Gauss's elimination. This involves algebraic
manipulations so that the result $X$ is rational. If $X$ is rational, then
multiplying by the product of the denominators of the entries, we obtain
an eigenvector with integer entries.

We now show that the eigenvalues of a graph Laplacian are
either integers or irrationals. 
We have the following rational root lemma on the roots of 
polynomials with integer coefficients, see for example
\cite{rational}
\begin{lemma} Rational root \\
Consider the polynomial equation
 $$a_n x^n + a_{n-1} x^{n-1} + \dots +a_0=0$$
where the coefficients $a_i$ are integers. Then, any rational
solution $x=p/q$, where $p,q$ are relatively prime is such that
$p$ divides $a_0$ and $q$ divides $a_n$ .
\end{lemma}

A consequence of this is 
\begin{theorem}
\label{eintirra}
The eigenvalues of a graph Laplacian are either integers or irrationals.
\end{theorem}	

\begin{proof}
Consider the equation associated to the characteristic 
polynomial associated to the graph Laplacian, it has the form
$$a_n x^n + a_{n-1} x^{n-1} + \dots +a_1 x, $$
because the graph is connected so that there is only one $0$
eigenvalue. 
Assume that the eigenvalue is of the form $x=p/q$ with $p,q$ are relatively prime integers. Then from the lemma above, $p$ divides
$a_0$ and $q$ divides $a_n$. Since $a_n= \pm 1$, $q=1$
so that $x=p$ is an integer.  
\end{proof}

The fact that some graphs have integer spectrum was discussed
by Grone and Merris \cite{gm94}. Many of their results are inequalities
for $\lambda_2$ and $\lambda_{n-1}$. Our results complement their
approach.

\section{Special graphs}

\subsection{Cliques and stars}

The clique $K_n$ has
eigenvalue $n$ with multiplicity $n-1$  and eigenvalue $0$.
The eigenvectors for eigenvalue $n$  can be
chosen as $v^k = e^1 -e^k,~~ k=2,\dots,n$. To see this note that
$$ L = n I_n - \mathbf{1}, $$
where $I_n$ is the identity matrix of order $n$
and $\mathbf{1}$ is the $(n,n)$ matrix where all elements are $1$.

A star of $n$ vertices $S_n$ is a tree such that one vertex , say vertex 1,
is connected to all the others.
For a star $S_n$, the eigenvalues and eigenvectors are
\begin{itemize}
\item $\lambda = 1$ multiplicity $n-2$ , eigenvector $e^2-e^k,~~k=3,\dots,n$
\item $\lambda = n$ multiplicity $1$ , eigenvector $(n+1)e^1 -\sum_{k=2}^n e^k$
\item $\lambda = 0$ multiplicity $1$ , eigenvector ${\hat 1}$
\end{itemize}

\subsection{Bipartite and multipartite graphs}

Consider a bipartite graph $K_{n_1,n_2}$.
The Laplacian is
\be\label{bipart}
L = 
\begin{pmatrix} 
n_2 & 0   &\dots &0     & -1 &\dots  &  & -1 \\ 
0   & n_2 & 0    &\dots & -1 &\dots  &  & -1   \\ 
\dots & \dots & \dots    &\dots & \dots &\dots &  & \\ 
0   & \dots & 0    & n_2 & -1 &\dots &   &  -1   \\ 
 -1 &\dots  &  & -1   & n_1 & 0  &\dots & 0 \\
 -1 &\dots  &  & -1   & 0 & n_1  &\dots & 0 \\
\dots   &  &     & & & \dots &\dots & \dots   \\ 
 -1 &\dots  &  & -1   & 0 & 0  &\dots & n_1
 \end{pmatrix} ,
\ee
where the top left bloc has size $n_1 \times n_1$, and the
bottom right bloc $n_2 \times n_2$.
The eigenvalues with their multiplicities denoted as exponents are 
$$ 0^1,~~ n_1^{n_2-1},~~ n_2^{n_1-1},~~ (n1+n2)^1  .$$
Eigenvectors for $n_1$ can be chosen as 
$e^{n_1+1}-e^i~~(i=n_1+2,\dots,n_1+n_2)$.
The eigenvector for $n=n_1+n_2$ is
$(1/n_1,\dots,1/n_1,-1/n_2,\dots,-1/n_2)^T$.

Similarly, the spectrum of a multipartite graph $K_{n_1,n_2,\dots n_p}$
is
$$ 0^1,~~ (n-n_1)^{n_1-1}, ~~ (n-n_2)^{n_2-1},\dots,~~(n-n_p)^{n_p-1},~~n^p.$$
The eigenvectors associated to $n-n_1$ are composed of $1$ and $-1$ in
two vertices of part 1 padded with zeros for the rest.

\subsection{Cycles}

For a cycle, the Laplacian is a circulant matrix, therefore its spectrum
is well-known. The eigenvalues are
\be\label{val_cyc}
\mu_k = 4 \sin^2 \left [ {(k-1) \pi \over n}  \right  ] ,~~k=1, \dots,n . \ee
They are associated to the complex eigenvectors $v^k$ whose components are 
\be\label{vec_cyc} v_j^k = \exp{ \left [ i(j-1)(k-1) 2 \pi \over n \right ] }~~ ,j=1, \dots n  .\ee 
The real eigenvectors $ w^k, ~ x^k$ are,
\begin{eqnarray}
 w^k =(0,~\sin (a_k),~ \sin (2 a_k),~  \dots,~ \sin ((n-1) a_k) )^T ,\\
 x^k =(1,~ \cos (a_k),~ \cos (2 a_k),~  \dots,~ \cos ((n-1) a_k) )^T ,\\
a_k = {2(k-1)  \pi \over n}
\end{eqnarray}
Ordering the eigenvalues, we have
\begin{eqnarray}
\lambda_1=\mu_1=0, \\
\lambda_2=\lambda_3=\mu_2, \\
\lambda_{2k}=\lambda_{2k+1}=\mu_{k+1},\\
\dots
\end{eqnarray}
For $n=2p+1$
$$\lambda_{2p}=\lambda_{2p+1}=\mu_{p+1}$$
For $n=2p$
$$\lambda_{2p}=\mu_{p}=4$$
is an eigenvalue of multiplicity 1; an eigenvector
is $(1,-1,\dots, 1,-1)^T$.
In all other cases, the eigenvalues have multiplicity two
so that all vertices are soft nodes.

Remark that the maximum number of 0s is $n/2$. To see this, note that
if two adjacent vertices 
have value 0 then their neighbors in the cycle must have 0 as well 
and we only have 0s
, but the null vector is not an eigenvector. 
This means that we have 
at most $n/2$ 0s.  This bound is reached for $n$ even.

\subsection{Chains}

For chains $C_n$, there are only single eigenvalues, they are 
\cite{edwards13} 
\be\label{vp_chemin}
\lambda_k = 4 \sin^2 ( {\pi (k-1) \over 2 n}) ~~ , k=1,\dots,n . \ee
The eigenvector $v^k$ has components
\be\label{vec_chemin}
v_j^k = \cos{ \left ( {\pi (k-1) \over n} (j-{1 \over 2})  \right ) }~~ ,j=1, \dots n  .\ee
Obviously the cosine is zero if and only if:
\be\label{cos0}(k-1)(2j-1) = n(1+2m),\ee
where $m$ is an integer.
There is no solution for $n=2^\alpha$, for $\alpha$ a positive integer.
Apart from this case, there is always at least one soft node. If
$n$ is a prime number, the middle vertex $j = (n+1)/2$ is the
only soft node. For $k$ odd, all vertices $j$ such that $2j-1$ divides $n$
have a zero value, including the middle vertex.


For $n$ odd, chains and cycles share $(n-1)/2$ eigenvalues
and eigenvectors. To see this consider a chain with $n=2p +1$.
All $k= 2 q +1 $ give a chain eigenvalue 
$\lambda_k =  4 \sin^2 ( {\pi q \over  2 p +1 })$
that is also a cycle eigenvalue. 
The eigenvector components $v_j^q$ are such that
$v_1^q = v_{2 p +1}^q$.

\section{Transformations preserving eigenvalues}

In this section, we present 
four main transformations of graphs such that one eigenvalue
is preserved. These are the link between two vertices, 
the articulation, the soldering and the contraction/expansion. 
The first three transformations are in the literature in a 
general form; we choose to present them in their most elementary
form. \\
Furthermore, these transformations will all be unary, 
they act on a single graph. Binary transformations can 
be reduced to unary transformations for non connected graphs.

Using these transformations
we can generate new graphs that have a soft node, starting from minimal graphs
having soft nodes. 

\subsection{Link between two equal vertices}

An important theorem due to Merris \cite{merris} connects  equal component vertices.
\begin{theorem}
{\bf Link} between two vertices : Let $\lambda$ be an eigenvalue of $L(G)$ for 
an eigenvector $x$.
If $x_i =x_j$ then $\lambda$ is an eigenvalue of $L(G')$ for $x$ where 
the graph $G'$ is obtained
from $G$ by deleting or adding the edge $e=ij$.
\end{theorem}
This transformation preserves the eigenvalue and eigenvector. It applies to
multiple graphs. 
Fig. \ref{l2s} shows examples of the transformation. 
\begin{figure} [H]
\centerline{
\epsfig{file=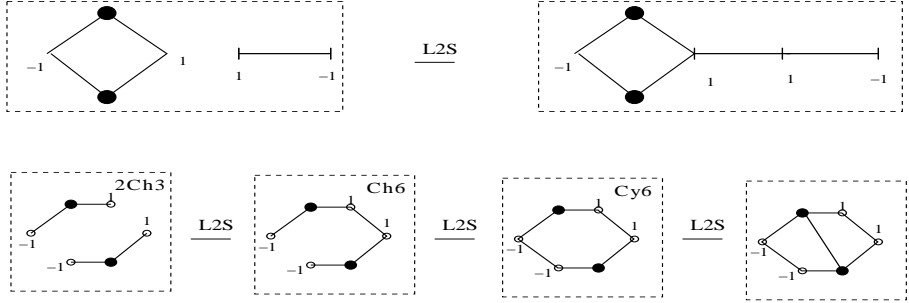,height=4 cm,width=12 cm,angle=0}
}
\caption{Example of the transform : link between two equal vertices.}
\label{l2s}
\end{figure}

We have the following corollary of the theorem.
\begin{theorem}
Let $\lambda$ be an eigenvalue of two graphs $G_1$ and $G_2$ for
respective eigenvectors $x^1,~x^2$ with two vertices $i,j$, 
such that $x_i^1 \neq 0$ or $x_j^2 \neq 0$ . Then the graph 
$G (V_1 \cup V_2, E_1 \cup E_2 \cup ij ) $ affords the eigenvector 
$y= x^2_j \begin{pmatrix} x^1 \cr 0 \end{pmatrix} + x_i^1  \begin{pmatrix} 0 \cr  x^2 \end{pmatrix} $ for $\lambda$.
\end{theorem}
This allows to generate many more graphs that have an eigenvalue $\lambda$.

\subsection{Articulation}

An elementary transformation inspired by Merris's principle
of reduction and extension \cite{merris}
is to add a soft node to an existing soft node.
This does not change the eigenvalue. We have the following 
result.
\begin{theorem}
{\bf Articulation (A)} : Assume a graph $G(V,E)$ with $n$ vertices where
$x$ is an eigenvector such that $x_i=0$ for an eigenvalue $\lambda$. 
Then, the extension $x'$ of $x$ such that $x'_{1:n}=x_{1:n}$ and
$x'_{n+1}=0$ is an eigenvector for $\lambda$ for the Laplacian $L(G')$
where $G'(V',E')$ such that $V'=V \cup (n+1)$ and $E'=E\cup i(n+1)$.
\end{theorem}
\begin{figure} [H]
\centerline{
\epsfig{file=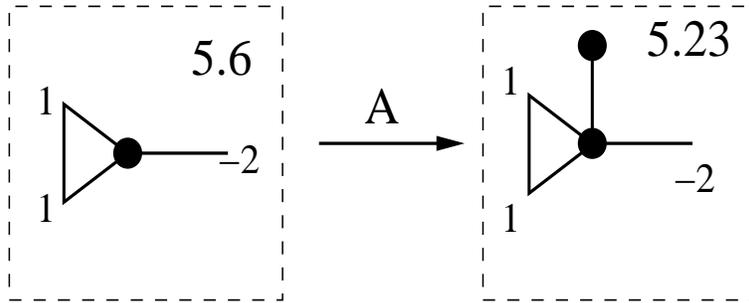,height=4 cm,width=10 cm,angle=0}
}
\caption{Example of the articulation property. The large dot
corresponds to a soft node.}
\label{f1}
\end{figure}
The general case presented by Merris \cite{merris} amounts to applying
several times this elementary transformation. \\
The transformation is valid for graphs with arbitrary weights
and the extended edges can have arbitrary weights.

Fig. \ref{f1} illustrates this property on the two graphs labeled 
$5.6$ and $5.23$ in the classification given in \cite{crs01}.
An immediate consequence of this elementary transform
is that any soft node can be extended into an arbitrarily large graph of
soft nodes while preserving the eigenvalue and extending the eigenvector
in a trivial way. Fig. \ref{f1a} shows two graphs that have the
same eigenvalue $\lambda=1$ and that are connected by the articulation 
transform.
\begin{figure} [H]
\centerline{
\epsfig{file=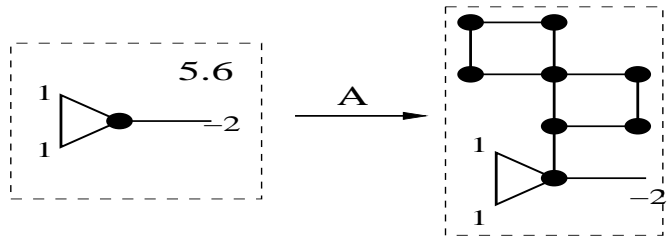,height=4 cm,width=10 cm,angle=0}
}
\caption{Two graphs connected by the articulation transform.}
\label{f1a}
\end{figure}

\subsection{Soldering}

A consequence of the contraction principle of Merris \cite{merris} is 
that coalescing two soft nodes of
a graph leaves invariant the eigenvalue. This is especially important
because we can "solder" two graphs at a soft node.
\begin{theorem}
{\bf Soldering } : Let $x$ be an eigenvector affording $\lambda$ 
for a graph $G$. Let $i$ and $j$ be two soft nodes 
without common neighbors. Let $G'$ be the graph obtained from
$G$ by contracting $i$ and $j$ and $x'$ be the vector obtained from
$x$ by deleting its $j$th component. Then $x'$ is an eigenvector
of $L(G')$ for $\lambda$.
\end{theorem}
\begin{figure} [H]
\centerline{
\epsfig{file=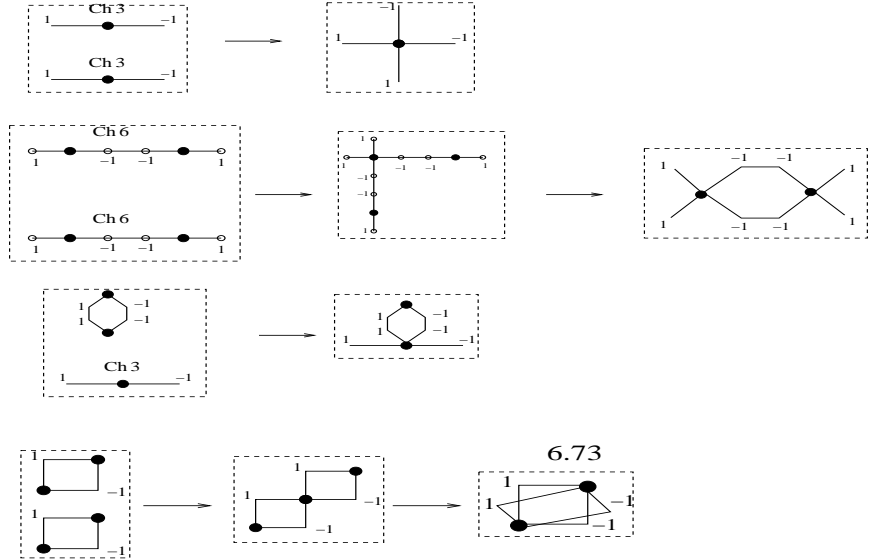,height=8 cm,width=12 cm,angle=0}
}
\caption{Examples of the soldering transform.}
\label{sold}
\end{figure}
This transformation is valid for graphs with arbitrary weights.

\subsection{Regular expansion of a graph}


We have the following theorem.
\begin{theorem}
Let $x$ be an eigenvector of a graph $G$ for $\lambda$ 
and let $i$ be a vertex connected only to $p$ soft nodes.
Let $G'$ be the graph obtained from $G$ by replacing $i$ by
a $d$-regular graph whose $k$ vertices are all connected to the
$p$ soft nodes. Then $\lambda=p$ and an eigenvector $x'$ of
$G'$ is formed by assigning to the new vertices, the value 
$x'_j=x_i/k$.
\end{theorem}

\begin{proof}
Without loss of generality, we can assume that $i=n$ and that
the $p$ soft nodes are $n-p+1,\dots,n-1$.
We have 
$$\left ( \begin{matrix}
\dots & \dots  & \dots & \dots  & \dots & \dots & \dots \cr
\dots & \dots  & \dots & \dots  & \dots & \dots & \dots \cr
\dots & \dots  & \dots & \dots  & \dots & \dots & \dots \cr
0 & \dots & 0 & -1     & \dots & -1 & p \cr
\end{matrix}  \right) 
\left ( \begin{matrix}
\dots \cr
0 \cr
0 \cr
x_n \cr
\end{matrix}  \right)
=
\lambda \left ( \begin{matrix}
\dots \cr
0 \cr
0 \cr
x_n \cr
\end{matrix}  \right)
$$
The $n$th line reads
$$ p x_n = \lambda x_n$$
so that $\lambda = p$.
The $n-1$th line reads
$$\alpha + (-1) x_n = p x_{n-1}=0$$
where $\alpha$ is the sum of the other terms.

Let us detail the eigenvector relation for the Laplacian for $G'$.
Consider any new vertex $j$ linked to the $p$ soft nodes and to $d$ 
new nodes. The corresponding line of the eigenvector relation 
for the Laplacian for $G'$ reads
$$ (d+p) x'_j + \sum_{i \sim j, i\ge n} (-1) x'_i = \lambda' x'_j .$$
This implies
$$(d+p-\lambda') x'_j = \sum_{i \sim j, i\ge n} x'_i .$$
An obvious solution is
$$\lambda' = \lambda = p,~~~x'_i =x'_n ~~\forall i\ge n+1. $$ 
The value $x'_n$ is obtained by examining line $n-1$. We have 
$$\alpha + \sum_{i=n}^{n-k-1}(-1)x_i'=0$$
so that $$x'_n = {x_n \over k}.$$
In fact, we can get all solutions by satisfying the two conditions
\be\label{cond_contrac} 
\forall j\ge n ~~ d x'_j = \sum_{i \sim j}x'_i,~~ x_n = \sum_{i\ge n} x'_i .\ee
\end{proof}

Fig. \ref{contrac} shows 
examples of expansion from a single soft node for different values of $d$.
Here the eigenvalue is $1$.
Fig. \ref{excontrac} shows examples of expansion from two soft nodes.
The eigenvalue is $2$.
For $d=2$, the values at the edges at the bold edges are such that
their sum is equal to $1$. 
For $d=2$, the values at
the triangle are all equal to $t$, the same holds for the square with a value
$s$. These values verify $3 t + 4 s = 1$.
\begin{figure} [H]
\centerline{
\epsfig{file=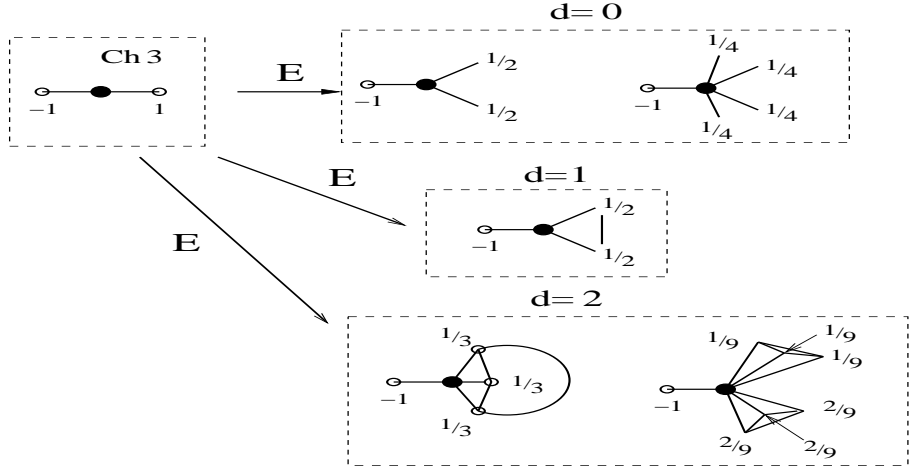,height=6.2 cm,width=12 cm,angle=0}
}
\caption{Examples of expansion from a single soft node.}
\label{contrac}
\end{figure}
\begin{figure} [H]
\centerline{
\epsfig{file=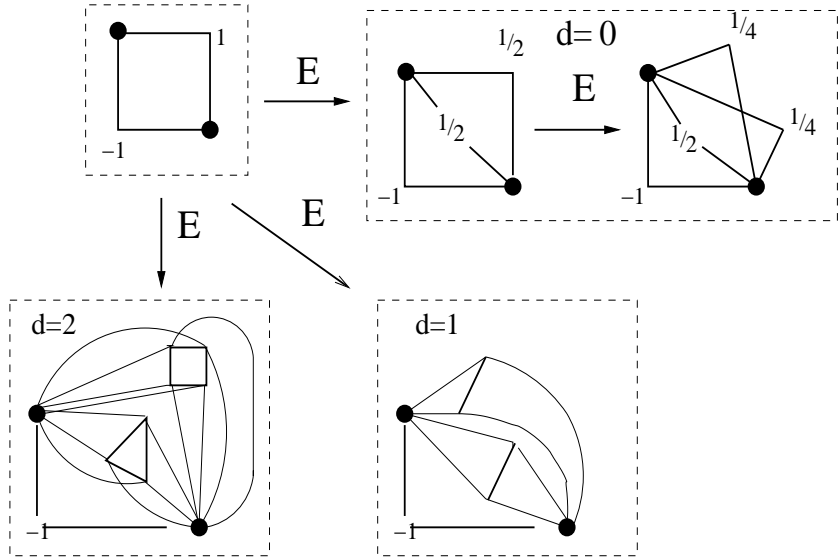,height=8 cm,width=12 cm,angle=0}
}
\caption{Examples of expansion from two soft nodes. For $d=2$, the values at
the triangle are all equal to $t$, the same holds for the square with a value
$s$. These values verify $3 t + 4 s = 1$.}
\label{excontrac}
\end{figure}

\subsection{Replace coupling by square}

We have the following transformation that leaves the
eigenvalue unchanged \cite{ckk18a}.
\begin{theorem}\textbf{(Replace an edge by a soft square)}\\
\label{gadget}
Let ${x}$ be an eigenvector of the Laplacian of a graph
${G}$ for an eigenvalue $\lambda$.
Let ${G}'$ be the graph obtained from ${G}$
by deleting a joint $ij$ such that ${x}_i=-{x}_j$ and
adding two soft vertices 
$k,l \in {V}({G}')$ for the extension
${x}'$ of ${x}$
(i.e. ${x}'_m={x}_m$ for $m\in {V}({G})$
and ${x}'_k={x}'_l=0$)
and the four edges $ik,kj,il,lj$.
Then, ${x}'$ is an eigenvector of the Laplacian of 
${G}'$ for the eigenvalue $\lambda$.
\end{theorem}
This result was proved in \cite{ckk18a} for a graph with weights
1. Here we generalize it to a graph with arbitrary weights.

\begin{proof}

The eigenvalue relation at vertex $i$ reads
$$(d_i - \lambda) x_i = \sum_{m\sim i, m \neq j} w_{i,m} x_m + w_{i,j} x_j $$
Since $x_i=-x_j$, this implies
$$(d_i +w_{i,j} - \lambda) x_i = \sum_{m\sim i, m \neq j} w_{i,m} x_m . $$
Introducing the two new vertices $k,l$ such that ${x}'_k={x}'_l=0$ 
connected to $i$ by edges of weights $w_{i,k}= \alpha w_{i,j}, ~~
w_{i,l}= (1-\alpha) w_{i,j}$, the relation above leads to
$$(d_i + w_{i,k} + w_{i,l} - \lambda) x'_i = \sum_{m\sim i} w_{i,m} x'_m + w_{i,k} x'_k + w_{i,l} x'_l , $$
which shows that $x'$ is eigenvector of the new graph.

\end{proof}

See Fig. \ref{esqua} for an illustration of the theorem.

\begin{figure} [H]
\centerline{
\epsfig{file=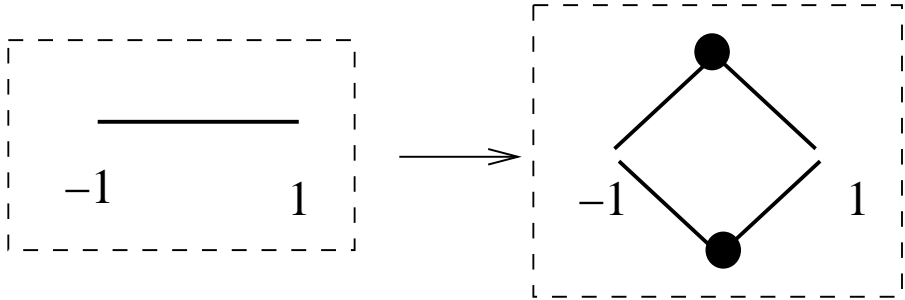,height=4 cm,width=12 cm,angle=0}
}
\caption{Replacement of coupling by a square, in both cases the eigenvalue
is $\lambda=2$.}
\label{esqua}
\end{figure}

\section{Transversality : change of eigenvalue}

Here we present operators that change the eigenvalue of a graph
Laplacian in a predictable
way. The operators shift the eigenvalue $\lambda$ to 
$\lambda+1$ for the first two and $\lambda+2$ for the third one.
At the end of the section we introduce the
eigenvalue of a product graph. 

\subsection{Inserting soft nodes}
\begin{theorem}
Let $x$ be an eigenvector of a graph $G$ with weights 1 for $\lambda$.
Assume we can pair the non zero components of $x$ as $\{i,j\}$
where $x_i=-x_j$ non zero.
Let $G'$ be the graph obtained from $G$ by including $k$
soft nodes between each pair $\{i,j\}$. The vector $x'$ so
obtained is an eigenvector of the Laplacian of $G'$ for
eigenvalue $\lambda +k$. 
\end{theorem}

\begin{proof}
Let $i,j \in V(G)$ be a pair such that $x_i=-x_j$. The 
eigenvector equation reads
$$d_i x_i -\sum_{m \sim i} x_m = \lambda x_i .$$
Introducing $k$ new vertices $x'_p =0,~~p=1,\dots k$
we can write the relation as
$$(d_i+k) x'_i -\sum_{m \sim i} x'_m = (\lambda+k) x'_i .$$
This shows that $x'$ is an eigenvector for the new graph.
\end{proof}
\begin{figure} [H]
\centerline{
\epsfig{file=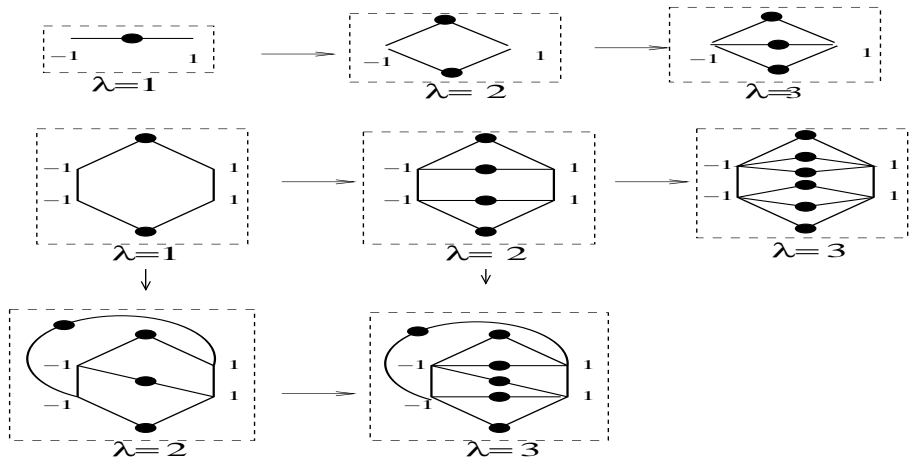,height=6 cm,width=12 cm,angle=0}
}
\caption{Example of the action of inserting a soft node.}
\label{lad}
\end{figure}
Fig. \ref{lad} shows
an example of the action of inserting a soft node.

When the graph is weighed, the result is still valid. Consider that
we add only one soft vertex connected to $i$ by a weight $w_{i,k}$. 
The eigenvalue of the new graph is $\lambda + w_{i,k}$. \\
This can transform a graph with an integer eigenvalue to a graph
with an irrational eigenvalue.

\subsection{Addition of a soft node} 
Connecting a soft node to all the vertices of 
a graph augments all the non zero eigenvalues by 1.
This result was found by Das \cite{das04}. We recover it here
and present it for completeness.
\begin{theorem}
{\bf Addition of a soft node} : Let $G(V,E)$ be a graph affording an eigenvalue 
$\lambda \neq 0$ for an eigenvector $x$. Then the new graph $G'$ obtained by adding
a node connected to all the nodes of $G$ has eigenvalue $\lambda +1$
for the eigenvector $x'$ obtained by extending $x$ by a zero component.
\end{theorem}
See Fig. \ref{add} for examples.

\begin{proof}
Assume $\lambda$ to be an eigenvalue with eigenvector $v$ for
the Laplacian $L(G)$ of a graph $G$ with $n$ vertices. Now add an
extra vertex $n+1$ connected to all vertices of $G$ 
and form $L(G \cup \{n+1\})$.
We have the following identity
$$ \begin{pmatrix}
             & | & -1 \\
L(G) + I_n   & | & -1 \\
             & | & -1 \\
-------------& | &---- \\
-1 -1, \dots -1& | & n 
\end{pmatrix} 
\begin{pmatrix} 
              \\
 v  \\
              \\
-- \\
0
\end{pmatrix} =
(\lambda +1 )
\begin{pmatrix} 
              \\
  v  \\
              \\
-- \\
0
\end{pmatrix} 
$$
which proves the statement.
\end{proof}
\begin{figure} [H]
\centerline{
\epsfig{file=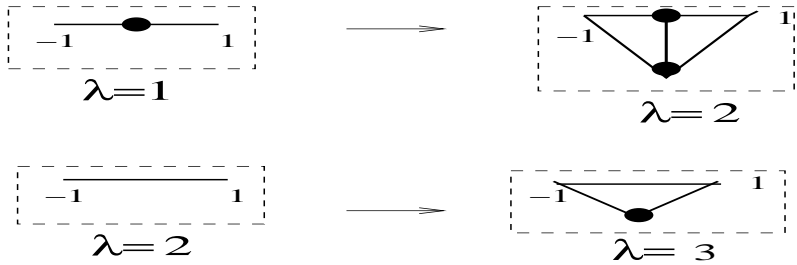,height=4 cm,width=12 cm,angle=0}
}
\caption{Examples of the addition of a soft node.}
\label{add}
\end{figure}

Important examples are the ones formed with the special graphs
considered above. There, adding a vertex to an $n-1$ graph, 
one knows explicitly $n-1$ eigenvectors
and eigenvalues.

The theorem 3.2 by Das \cite{das04} can be seen as a direct consequence
of adding a soft node and an articulation to a graph.

\subsection{Inserting a matching }

First we define perfect and alternate perfect matchings.
\begin{definition}[Perfect matching]
A perfect matching of a graph ${G}$ is a matching 
(i.e., an independent edge set)
in which every vertex of the graph is incident to exactly one edge of the matching.
\end{definition}
\begin{definition}[Alternate perfect matching]
An alternate perfect matching for a vector $v$ on the nodes of a graph ${G}$
is a perfect matching for the nonzero nodes such that edges $e_{ij}$ of the matching
satisfy $v_i=-v_j ~~ ( \neq 0)$.
\end{definition}

We have the following result \cite{ckk18a} inspired by the alternating principle
of Merris \cite{merris}.
\begin{theorem}[\textbf{Add/Delete an alternate perfect matching}]
\label{alt}
Let $v$ be an eigenvector of $L({G})$ affording an eigenvalue $\lambda$.
Let ${G}'$ be the graph obtained from ${G}$
by adding (resp. deleting) an alternate perfect matching for $v$.
Then, $v$ is an eigenvector of $L({G}')$ affording the eigenvalue 
$\lambda +2$ (resp. $\lambda -2$).
\end{theorem}
This is a second operator which shifts eigenvalues
by $\pm 2$. Examples are given in Fig. \ref{lad2}.
\begin{figure} [H]
\centerline{
\epsfig{file=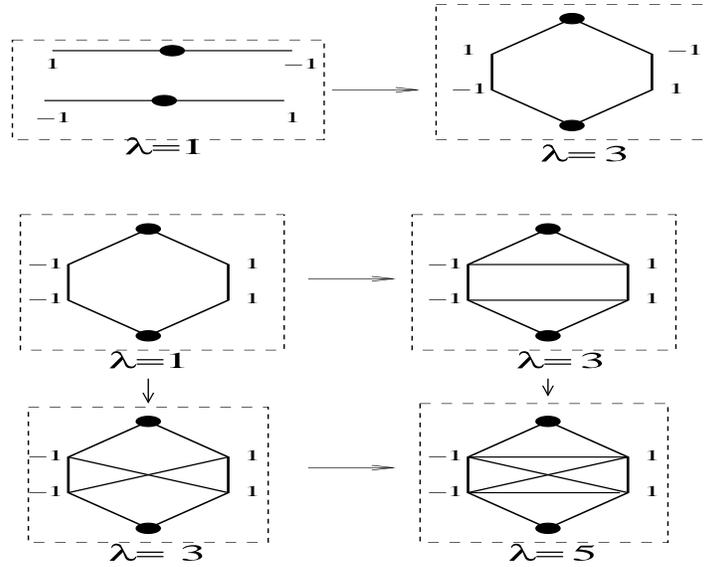,height=8.2 cm,width=12 cm,angle=0}
}
\caption{Examples of inserting a matching.}
\label{lad2}
\end{figure}

\subsection{Cartesian product}

The cartesian product $G\square H$ of two graphs $G = (V,E)$ and $H = ( W,F)$
has set of vertices $V \times W = \{ (v,w), v \in V, ~w \in W\}$.
It's set of edges is $\{\{(v_1,w_1),(v_2,w_2)\}\}$ such
that $v_1~v_2 \in V$ and $w_1w_2 \in W$.
We have the following result, see Merris \cite{merris}.
\begin{theorem}
If $x$ is an eigenvector of $G$ affording $\mu$ and
$y$ is an eigenvector of $H$ affording $\nu$, then 
the Kronecker product of $x$ and$y$ , $x \otimes y$
is an eigenvector of $G \square H$ for the eigenvalue $\mu+\nu$. 
\end{theorem}
Fig. \ref{cart} illustrates the theorem.

\begin{figure} [H]
\centerline{
\epsfig{file=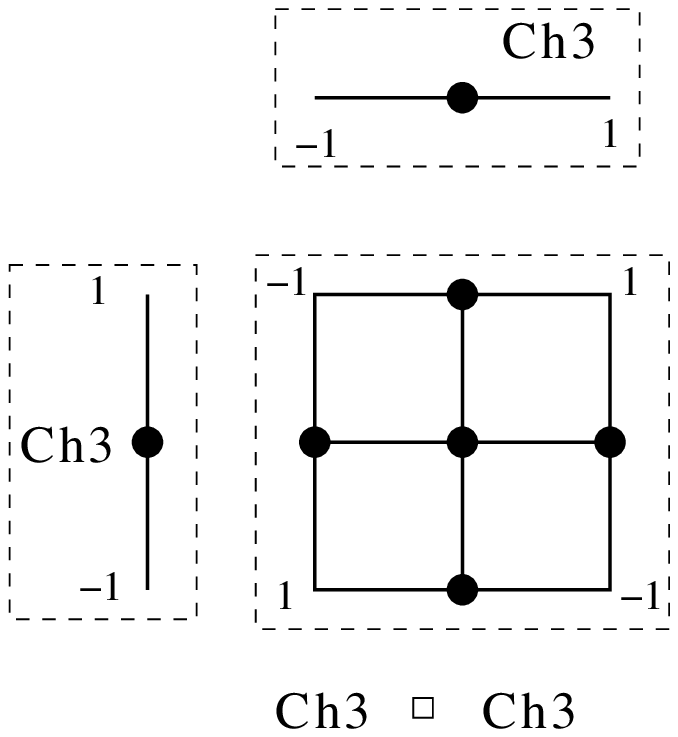,height=5 cm,width=6 cm,angle=0}
\hskip .3cm 
\epsfig{file=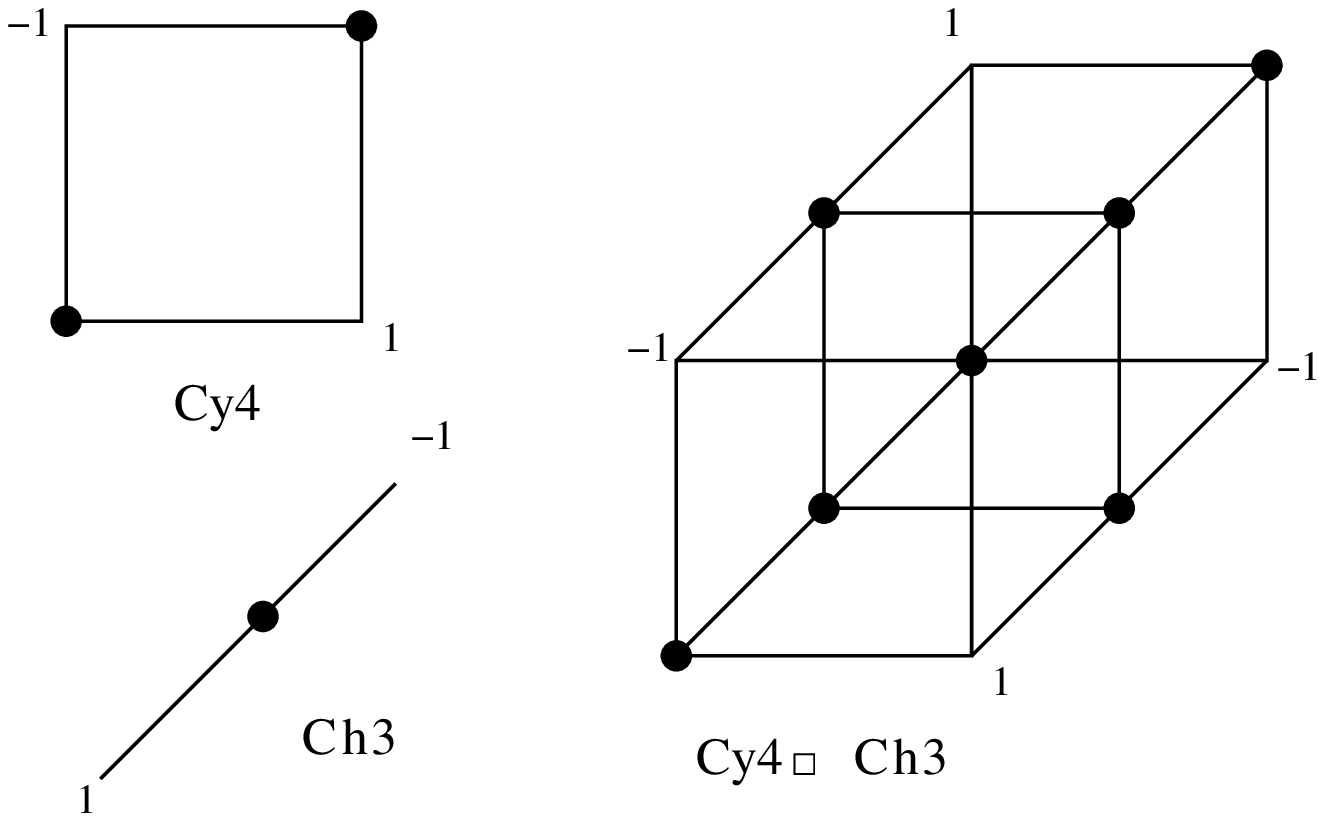,height=5 cm,width=6 cm,angle=0}
}
\caption{Cartesian product of two chains 3 (left) and of a
cycle 4 and a chain 3 (right).}
\label{cart}
\end{figure}

Important examples are the ones formed with the special graphs
considered above. There, one knows explicitly the eigenvectors
and eigenvalues. For example, the cartesian product
$C_n \times C_m$ of two chains $C_n$ and $C_m$ with $n$ and $m$
nodes respectively has eigenvalues
$$\lambda_{i,j} = \lambda_i + \lambda_j ,$$
where $\lambda_i$ (resp. $\lambda_j$) is an eigenvalue for $C_n$
(resp. $C_m$). The eigenvectors are
$$v^{i,j}=\cos [{\pi (i-1) \over  n}(p -{1\over 2}) ]
\cos [{\pi (j-1) \over  m}(q -{1\over 2}) ],$$
where $ i,p \in \{1,\dots ,n\},~~~j,q \in \{1,\dots ,m\}$.

\subsection{Graph complement}

We recall the definition of the complement of a graph $G$.
\begin{definition}[{\rm Complement of a graph }]
\label{complement}
Given a graph $G(V,E)$ with $n$ vertices, its complement $G^c$ is the
graph 
$G^c(V,E^c)$ where $E^c$ is the complement of $E$ in the
set of edges of the complete graph $K_n$. 
\end{definition}
We have the following property, see for example \cite{crs01}.
\begin{theorem}
If $x$ is an eigenvector of a graph $G$ with $n$ vertices
affording $\lambda \neq 0$, then
$x$ is an eigenvector of $G_c$ affording $n-\lambda$.
\end{theorem}

An example is shown in Fig. \ref{comp}. 
The eigenvalues and eigenvectors are given in table \ref{tcomp}.
\begin{figure} [H]
\centerline{
\epsfig{file=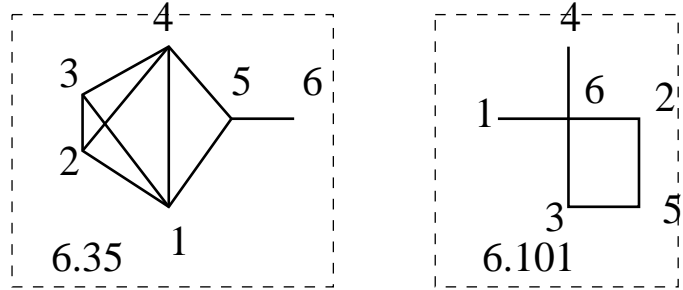,height=5 cm,width=10 cm,angle=0}
}
\caption{Graph 6.35 (left) in the classification \cite{crs01} and its
complement 6.101 (right). }
\label{comp}
\end{figure}
\begin{table} [H]
\centering
\begin{tabular}{|l|c|c|c|c|c|r|}
\hline
6.35   &5.2361  &  5.   &  4     & 3        &  0.7639     & 0 \\
       &        &       &        &          &         &  \\ \hline
6.101  & 0.7639 &  1    &  2     & 3        &  5.2361     & 0 \\
       &        &       &        &          &         &   \\ \hline
  &   0.51167 & 0.70711 & 0.     &  0.18257 &-0.19544 & 0.40825 \\
  & -0.31623  & 0.      & 0.70711& -0.36515 &-0.31623 & 0.40825 \\
  & -0.31623  & 0.      &-0.70711& -0.36515 &-0.31623 & 0.40825 \\
  & 0.51167   &-0.70711 & 0.     &  0.18257 &-0.19544 & 0.40825 \\
  & -0.51167  & 0.      & 0.     &  0.73030 & 0.19544 & 0.40825 \\
  & 0.12079   & 0.      & 0.     & -0.36515 & 0.82790 & 0.40825 \\ \hline
\end{tabular}
\caption{\small\em Eigenvalues (top lines) and eigenvectors for
the two complementary graphs 6.35 and 6.101 shown in Fig. \ref{comp}}
\label{tcomp}
\end{table}

Many times, $G_c$ is not connected. An example where $G_c$
is connected is the cycle 6....

\section{$\lambda$-soft graphs }

\subsection{Definitions and properties }

We introduce the notions of $\lambda$, $\lambda$ soft and $\lambda$ soft minimal
graphs. The transformations of the previous section will enable us
to prove the relation between these two types of graphs.
\begin{definition}
A graph $G$ affording an eigenvector $X$ for an eigenvalue
$\lambda$ is $\lambda$. 
\end{definition}
\begin{definition}
A $\lambda$ graph $G$ affording an eigenvector $X$ for the eigenvalue 
$\lambda$ is $\lambda$ soft if one of the entries of $X$ is zero.
\end{definition}
\begin{definition}
A graph $G$ affording an eigenvector $X$ for an eigenvalue 
$\lambda$ is $\lambda$ minimal if it is $\lambda$ 
and minimal in the sense of inclusion. 
\end{definition}
Clearly, for a given $\lambda$, there is at least one $\lambda$ minimal graph.
As an example the 1 soft minimal graph is shown below. 
\begin{figure} [H]
\centerline{ \epsfig{file=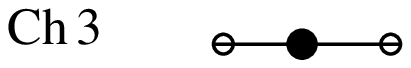,height=1 cm,width=6 cm,angle=0} }
\end{figure}

\subsection {$\lambda$ subgraph}

In the following section, we study the properties of a $\lambda$ subgraph $G$
included in a $\lambda$ graph $G"(V",E")$. 
Consider two graphs $G(V,E)$ with $n$ vertices and $G'(V"-V,E')$
such that $E$ only connects elements of $V$ and $E'$ only
connects elements of $V'$.
Assume two graphs $G(n)$ and $G'(n')$ are included in a large graph $G"$
and are such that $G(V,E)$
Assume $p$ vertices of $G$ are linked to $p'$ vertices of $G'$.
We label the $p$ vertices of $G$, $n-p+1, \dots, n$ and the
$p'$ vertices of $G'$, $1,\dots,p'$.
We have
\begin{eqnarray}  
\label{eigsmall} L X = \lambda X, \\
\label{eiglarge} 
L" X" = \lambda X" , 
\end{eqnarray}
where $L"$ is the graph Laplacian for the large graph $G"$;  $L"$ can be written
as 
$$L" = \begin{pmatrix} L & 0 \\ 0 & L'\end{pmatrix}
        + \begin{pmatrix} 0 & 0 & 0 & 0 \\ 0 & a & -b & 0 \\
0 & -b^T & c & 0 \\ 0 & 0 & 0 & 0 \\
\end{pmatrix}  . $$
A first result is
\begin{theorem} The square matrix $\delta = \begin{pmatrix} a  & -b\\
-b^T & c \\
\end{pmatrix}$ is a graph Laplacian.
\end{theorem}
\begin{proof}
The submatrices $a,b,c$ have respective sizes $a(p,p),~b(p,p'),~c(p',p')$,
$a$ and $c$ are diagonal and verify
\be\label{abc}
a_{ii}= \sum_{j=1}^{p'} b_{ij}, ~~c_{ii}= \sum_{j=1}^{p} b_{ji}  . \ee
In other words 
$$a {\hat 1}_p = b {\hat 1}_{p'},~~~c {\hat 1}_{p'} = b^T {\hat 1}_{p}, $$
where ${\hat 1}_{p}$ is a $p$ column vector of 1.
\end{proof}

At this point, we did not assume any relation between the eigenvectors
$X$ for $G$ and $X"$ for $G"$. We have the following 
\begin{theorem} The eigenvalue relations (\ref{eigsmall},\ref{eiglarge})
imply either $X=X"(1:n)$ or $X(1:n-p)=0$.
\end{theorem}
\begin{proof}
For $p=1$ and $\lambda$ a single eigenvalue, ${\rm rank}(L-\lambda I)=n-1$ so either $X=X"(1:n)$ or $X(1:n-1)=0$. \\
We admit the result for $p>1$.
\end{proof}

We can then assume that the eigenvectors of $L"$ have the form
$$L" \begin{pmatrix} X \\ X' \end{pmatrix}= \lambda \begin{pmatrix} X \\ X' \end{pmatrix} , $$  
where $L X = \lambda X$.
Substituting $L"$, we get
$$\lambda \begin{pmatrix} X \\ X' \end{pmatrix}
=  \begin{pmatrix} L & 0 \\ 0 & L'\end{pmatrix} \begin{pmatrix} X \\ X' \end{pmatrix}
+ \begin{pmatrix} 0 & 0 & 0 & 0 \\ 0 & a & -b & 0 \\
0 & -b^T & c & 0 \\ 0 & 0 & 0 & 0\\
\end{pmatrix} \begin{pmatrix} X \\ X' \end{pmatrix} .$$
Using the relation (\ref{eigsmall}) we obtain
\begin{eqnarray}
\begin{pmatrix} 0 & 0 \\ 0 & a \end{pmatrix} X
+ \begin{pmatrix} 0 & 0 \\ -b & 0 \end{pmatrix} X'=0,  \\
L' X' -\begin{pmatrix} 0 & b^T \\ 0 & 0 \end{pmatrix} X
+ \begin{pmatrix} c & 0 \\ 0 & 0 \end{pmatrix} X' = \lambda X' .
\end{eqnarray}
There are $p$ non trivial equations in the first matrix equation
and $p'$ in the second one. 
Using an array notation (like in Fortran), the system above can be written as
\begin{eqnarray}
a X(n-p+1:n)-b X'(1:p')=0, \\
-b^T X(n-p+1:n)+c X'(1:p') + (L'X')(1:p') =\lambda X'(1:p'), \\
(L'X')(p'+1:n') =\lambda X'(p'+1:n'), 
\end{eqnarray}
Extracting $X$ from the first equation, we obtain 
\be \label{gpg1} X(n-p+1:n)= a^{-1} b X'(1:p') , \ee 
and substituting in the second equation yields the closed system in $X'$
\begin{eqnarray}
\label{gpg2} (-b^T a^{-1} b +c)X'(1:p')  + (L'X')(1:p')=
\lambda X'(1:p'), \\
\label{gpg3} (L'X')(p'+1:n')=\lambda X'(p'+1:n'),
\end{eqnarray}
where we used the fact that the matrix $a$ of the degrees of the connections
is invertible by construction.

\begin{theorem} The matrix
$$\Delta \equiv -b^T a^{-1} b +c , $$
is a generalized graph Laplacian: it is a Laplacian of a weighted graph.
Its entries are rationals and not necessarily integers.  
\end{theorem}
\begin{proof}
To prove this, note first that $\Delta$ is obviously symmetric.
We have
$$\Delta {\hat 1}_{p'} = -b^T a^{-1} b {\hat 1}_{p'} + c {\hat 1}_{p'}
= -b^T a^{-1} a {\hat 1}_{p} + b^T {\hat 1}_{p}  = 0.$$
This shows that the each diagonal element of $\Delta$ is equal to the
sum of it's corresponding row so that $\Delta$ is a graph Laplacian.
\end{proof}
From theorem (\ref{eintirra}), the eigenvalues of $\Delta$
are integers or irrationals and correspond to eigenvectors with
integer or irrational components.

We then write equations (\ref{gpg2},\ref{gpg3}) as
\be\label{eigy} ({\bar \Delta} + L') X' = \lambda X' , \ee
where
$$ {\bar \Delta} = \begin{pmatrix} \Delta & 0 \\ 0 & 0 \end{pmatrix}$$
This is an eigenvalue relation for the graph Laplacian $({\bar \Delta} + L')$.
Four cases occur. 
\begin{itemize}
\item [(i)] $\lambda=0$ then $X'$ is a vector of equal components and $X$ also.
\item [(ii)] $\lambda\neq 0$ is an eigenvalue of $L'$. Then one has
the following 
\begin{theorem}\label{sliart} 
Assume a graph $G"$ is $\lambda$ for an eigenvector $X"=(X,X')^T$ 
and contains a $\lambda$ graph $G$ for the eigenvector $X$. 
Consider the graph $G'$ with vertices
$V(G")-V(G)$ and the corresponding edges in $G"$. \\
If $G'$ is $\lambda$ then $G"$ is obtained from $G$
using the articulation or link transformations. 
\end{theorem}
\begin{proof}
Since $\lambda\neq 0$ is an eigenvalue of $L'$, we can choose
$X'$ an eigenvector for $\lambda$ so that 
$L' X' = \lambda X'$, then $\Delta X' =0$. \\
A first possibility is $X'=0$, this corresponds to an articulation
between $G$ and $G'$. \\
If $X' \neq 0$, $L' X' = \lambda X'$, implies that $X'$ is not 
a vector of equal components so that $X' \notin {\rm Null} (\Delta)$.
The only possibility for $\Delta X' =0$ is $\Delta=0$ so that
$$c = b^T a^{-1} b . $$ 
The term $ (b^T a^{-1} b)_{ij}$ is
$$ (b^T a^{-1} b)_{ij} = \sum_{k=1}^p {b_{ki} b_{kj} \over a_{kk}} . $$
Since the matrix $c$ is diagonal, we have
$$\sum_{k=1}^p {b_{ki} b_{kj} \over a_{kk}} = 0, \forall i \neq j$$
Then $b_{ki} b_{kj} =0$ so that a vertex $k$ from $G$ is only connected to
one other vertex $i$ or $j$ from $G'$. Then $p=p'$.
This implies $a_{ii}=c_{ii}=1,  \forall i \in \{1,, \dots,p\}$.
The graphs $G$ and $G'$ are then connected by a number of edges
between vertices of same value.
\end{proof}

\item [(iii)] $\lambda\neq 0$ is not an eigenvalue of $L'$ and
 $L'$ and ${\bar \Delta}$ share a common eigenvector $X'$ for
eigenvalues $\lambda'$ and $\lambda-\lambda' >0$. \\
For $\lambda-\lambda'=1$, a possibility is to connect a soft node 
of $G$ to $G'$.
For $\lambda-\lambda'=p$ integer, a possibility is to connect $p$ 
soft nodes of $G$ to $G'$.\\
We conjecture that there are no other possibilities.

\item [(iv)] $\lambda\neq 0$ is not an eigenvalue of $L'$ and
$L'$ and ${\bar \Delta}$ have different eigenvectors.
Then there is no solution to the eigenvalue problem (\ref{eigy}). \\
To see this, assume the eigenvalues and eigenvectors of $L'$ and ${\bar \Delta}$ are
respectively $\nu_i, V^i$, $\mu_i,W^i$ so that
$$L' V^i = \nu_i V^i,~~{\bar \Delta} W^i = \mu_i W^i , ~~i=1,2,\dots n$$
The eigenvectors can be chosen orthonormal and we have $$ Q V = W Q $$
where $Q =(q_k^j)$ is an orthogonal matrix, $V$ and $W$ are the matrices
whose columns are respectively $V^i$ and $W^i$. We write
$$ W^j = \sum_k q_k^j V^k .$$

Assuming $X'$ exists, we can expand it as $X'= \sum_i \alpha_i V^i$ 
Plugging this expansion intro the relation
$({\bar \Delta} + L') X' = \lambda X'$ yields
$$\sum_i \left ( \alpha_i \nu_i V^i + \alpha_i \sum_j q_j^i \mu_j \sum_k q_k^j V^k\right ) = \sum_i \lambda \alpha_i \nu_i V^i$$
Projecting on a vector $V^m$ we get
$$ \alpha_m \nu_m  + \alpha_m \sum_j q_j^m \mu_j q_m^j  = \lambda \alpha_m \nu_m $$
A first solution is $\alpha_m=0, \forall m$ so that $X'=0$, an articulation.
If $\alpha_m \neq 0$ then we get the set of linear equations
linking the $\nu_i$ to the $\mu_i$.
$$ \sum_j q_j^m \mu_j q_m^j = (\lambda -1) \nu_m, ~~m=1,\dots n$$
Since $Q$ is a general orthogonal matrix, the terms
$q_j^m$ are irrational in general. Therefore we conjecture that there are
no solutions.
\end{itemize}

\subsection{Examples of $\lambda$ subgraphs}

Using simple examples, we illustrate the different scenarios
considered above.  
We first consider theorem (\ref{sliart}), see Fig. \ref{GGp}.
\begin{figure} [H]
\centerline{
\epsfig{file=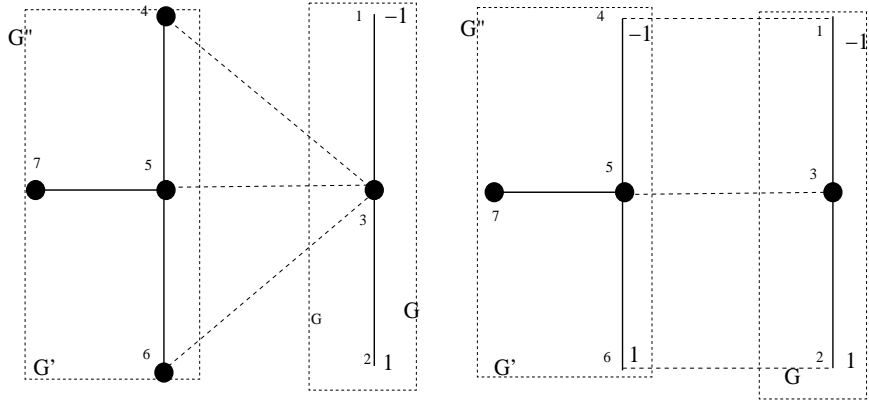,height=6 cm,width=12 cm,angle=0}
}
\caption{Two configurations where a graph $G$ is included in a 
larger graph $G"$ for the eigenvalue $1$.}
\label{GGp}
\end{figure}
Consider the configuration on the left of Fig. \ref{GGp}. We have
\be L = \begin{pmatrix} 
1 & -1 & 0   \\ 
0  &1 & 1  \\
-1 & -1 & 2  \\ 
\end{pmatrix},  ~~~~
L' = \begin{pmatrix}
1 & -1 & 0 & 0  \\ 
-1 & 3 & -1& -1  \\ 
0  & -1 & 1 & 0  \\ 
0  & -1 & 0 & 1  
\end{pmatrix}. \ee
Note that $L$ and $L'$ have 1 as eigenvalue.
Here $p=1,p'=3$ and
$$a=3, b=(1 ,1 ,1)^T, c = \begin{pmatrix}
1 & 0 & 0   \\
0 & 1 & 0  \\
0  &0 & 1
\end{pmatrix},$$
so that $$\Delta = \begin{pmatrix}
{2 \over 3} & -{1 \over 3} & -{1 \over 3} \\
-{1 \over 3} & {2 \over 3} & -{1 \over 3}  \\
-{1 \over 3} &-{1 \over 3} &  {2 \over 3}  
\end{pmatrix}.$$
The matrices ${\bar \Delta}$ and $L'$
have different eigenvectors for the same eigenvalue 1. 
Choosing $X'$ an eigenvector of $L'$ for the eigenvalue 1
yields $ {\bar \Delta}  X' = 0$.
The only solution is $X'=0$, this is an articulation.

For the configuration on the right of Fig. \ref{GGp} we have
$p=p'=3$. 
$$
a= \begin{pmatrix}
1 & 0 & 0   \\
0 & 1 & 0  \\
0  &0 & 1
\end{pmatrix} ,~~~
b = \begin{pmatrix}
1 & 0 & 0 & 0  \\
0 & 1 & 0 & 0 \\
0  &0 & 1 & 0
\end{pmatrix} ,~~~
c = \begin{pmatrix}
1 & 0 & 0   \\
0 & 1 & 0  \\
0  &0 & 1  
\end{pmatrix},
$$
so that $\Delta = \begin{pmatrix} 
0 & 0 & 0 \\
0 & 0 & 0 \\
0 & 0 & 0 
\end{pmatrix}.
$
We have 
\begin{eqnarray}
\label{link1} L X=  1 X, \\
\label{link2} L"  (X,X')^T= 1 (X,X')^T ,
\end{eqnarray}
where $X=(X_1,X_2,X_3)^T$ 
In this configuration, $X'$ is an eigenvector of $L'$
for the eigenvalue 1. 
and we have Link connections between $G$ and $G'$.

Finally, we show an example of case (iii) where $G,G"$ are 2 soft
and  $G'$ is 1 soft.
\begin{figure} [H]
\centerline{
\epsfig{file=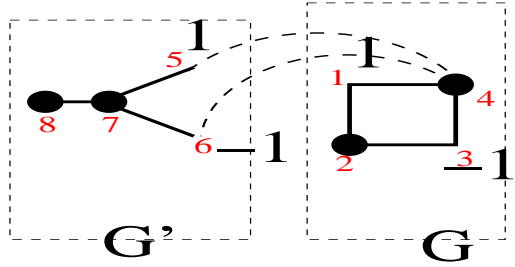,height=4 cm,width=8 cm,angle=0}
}
\caption{An example of case (iii) for eigenvalue $\lambda=2$.}
\label{sub4}
\end{figure}
We have to solve $({\bar \Delta} + L') X' = 2 X'$ where
$$L = \begin{pmatrix} 
2 & -1 & 0 & -1 \\
-1 & 2 & -1 & 0 \\
0 & -1 & 2 & -1 \\
-1 & 0 & -1 & 2 
\end{pmatrix}, ~~
L' = \begin{pmatrix}
1& 0& -1& 0 \\
0& 1& -1& 0  \\
-1& -1& 3& -1 \\
0 &  0& -1& 1 
\end{pmatrix}, ~~
{\bar \Delta} = \begin{pmatrix}
0.5  &  -0.5  &  0 \\
-0.5  &  0.5 &  0 \\
0  &  0  &  0 
\end{pmatrix} .$$
Note that the eigenvector $X'=(1,-1,0,0)^T$ is shared by $L'$
and ${\bar \Delta}$ so that $({\bar \Delta} + L') X' = 2 X'$.

The transformations introduced in the two previous sections enable us to link
the different members of a given class. To summarize, we have
\begin{itemize}
\item Articulation : one can connect any graph $G_2$ to the soft nodes
of a given graph $G_1$ and keep the eigenvalue. The new graph $G_1 \cup G_2$
has soft nodes everywhere in $G_2$.
\item Link : introducing a link between equal nodes does not change the
eigenvalue and eigenvector.
\item Contraction of a d-regular graph linked to a soft node. To have
minimal graphs in the sense of Link we need to take $d=0$.
\item Soldering : one can connect two graphs by contracting one or several
soft nodes of each graph.
\end{itemize}
In the next subsections we present a classification of small size 
$\lambda$ soft graphs for different $\lambda$s.

\subsection{$1$-soft graphs }

\begin{figure} [H]
\centerline{
\epsfig{file=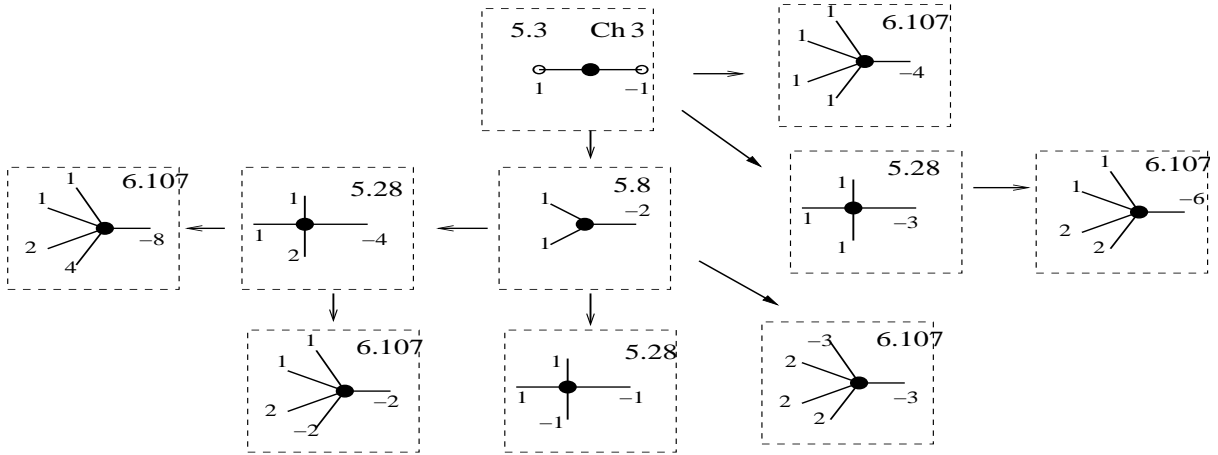,height=6 cm,width=16 cm,angle=0}
}
\caption{$1_s$ graphs: graphs generated by expansion. }
\label{contrac1}
\end{figure}
Fig. \ref{contrac1} shows some of the 1s graphs generated by expansion.
Note the variety of possibilities. 
\begin{figure} [H]
\centerline{
\epsfig{file=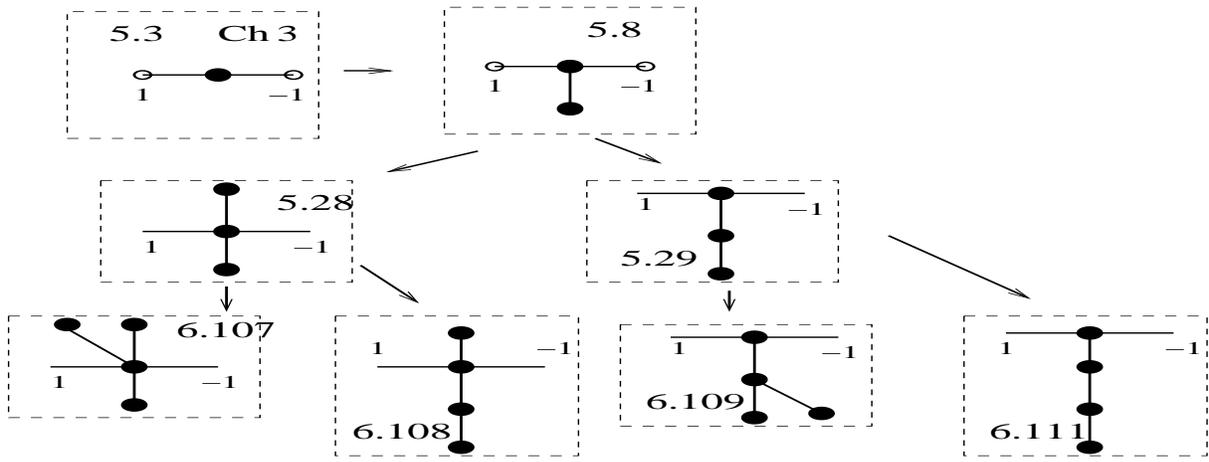,height=6 cm,width=16 cm,angle=0}
}
\caption{$1_s$ graphs: graphs generated by articulation}
\label{arti1}
\end{figure}
Fig. \ref{arti1} shows some of the 1s graphs generated by articulation.
The $1 ,  0,  -1$ configuration remains clearly visible.

\begin{figure} [H]
\centerline{
\epsfig{file=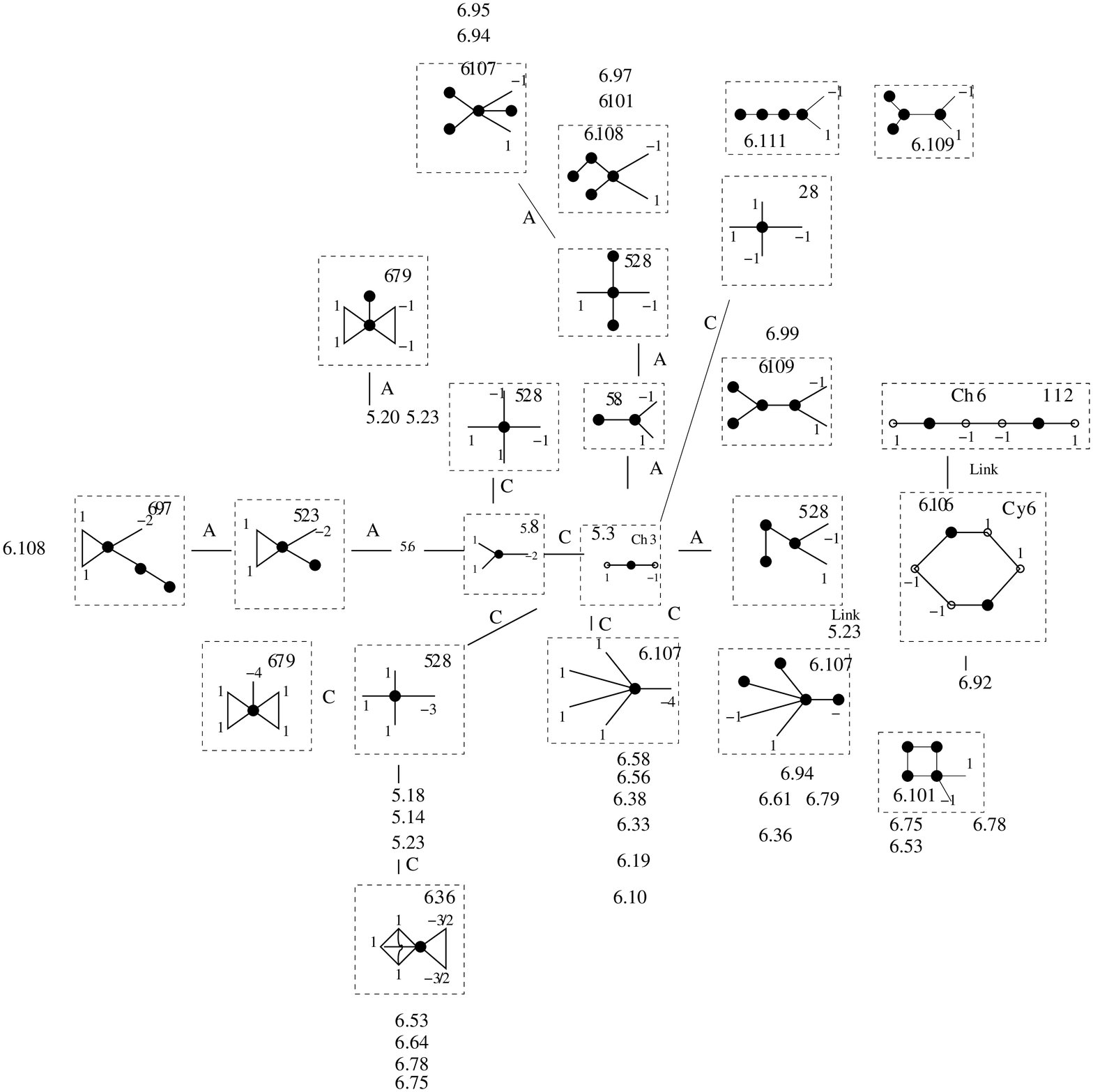,height=12 cm,width=17 cm,angle=0}
}
\caption{$1$-soft graphs. The soft nodes are in boldface. We only present
symmetric expansions so that links are possible.}
\label{l1}
\end{figure}
Fig. \ref{l1} shows the 1s graphs with at most 6 vertices. Notice
how they are linked by articulation (A), expansion/contraction (C) and links
and can all be obtained from the graph 5.3 (chain 3).
The connection $Ch3$ - $28$ is a contraction of two $Ch3$ chains.
Connecting two 3 chains $Ch3$ with an Link transformation we obtain 
a chain 6 $Ch6$. One can also go from $Ch6$ to 23 by soldering
the two soft nodes.

\subsection{$2$-soft graphs }

Fig. \ref{contrac2} shows some of the 2s graphs generated by expansion of
the 5.7 graph. 
\begin{figure} [H]
\centerline{
\epsfig{file=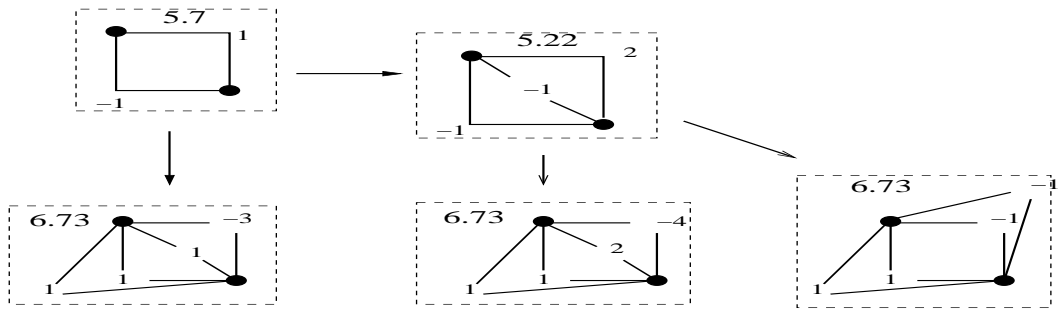,height=4 cm,width=14 cm,angle=0}
}
\caption{$2_s$ graphs: graphs generated by expansion. }
\label{contrac2}
\end{figure}
Similarly Fig. \ref{arti2} shows some of the 2s graphs generated by 
articulation from the same graph.
\begin{figure} [H]
\centerline{
\epsfig{file=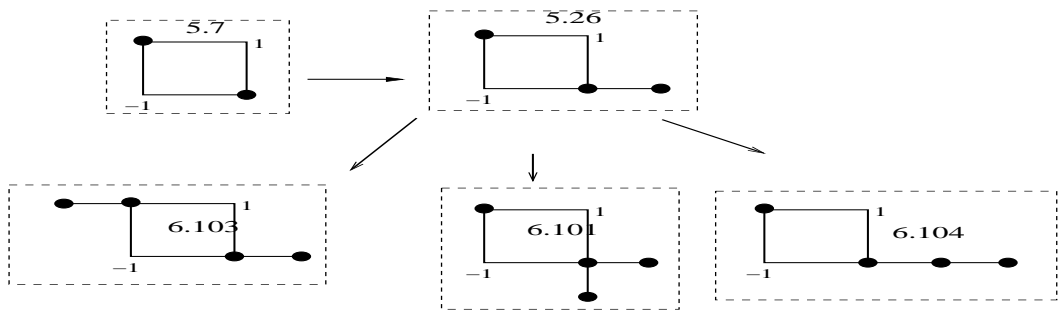,height=4 cm,width=14 cm,angle=0}
}
\caption{$2_s$ graphs: graphs generated by articulation}
\label{arti2}
\end{figure}
Fig. \ref{l2} shows all 2s graphs with at most 6 vertices. We included
graph 5.1 because with a link it gives configuration 6.104.
Notice how all graphs can be generated from 5.5 and 5.1.
\begin{figure} [H]
\centerline{
\epsfig{file=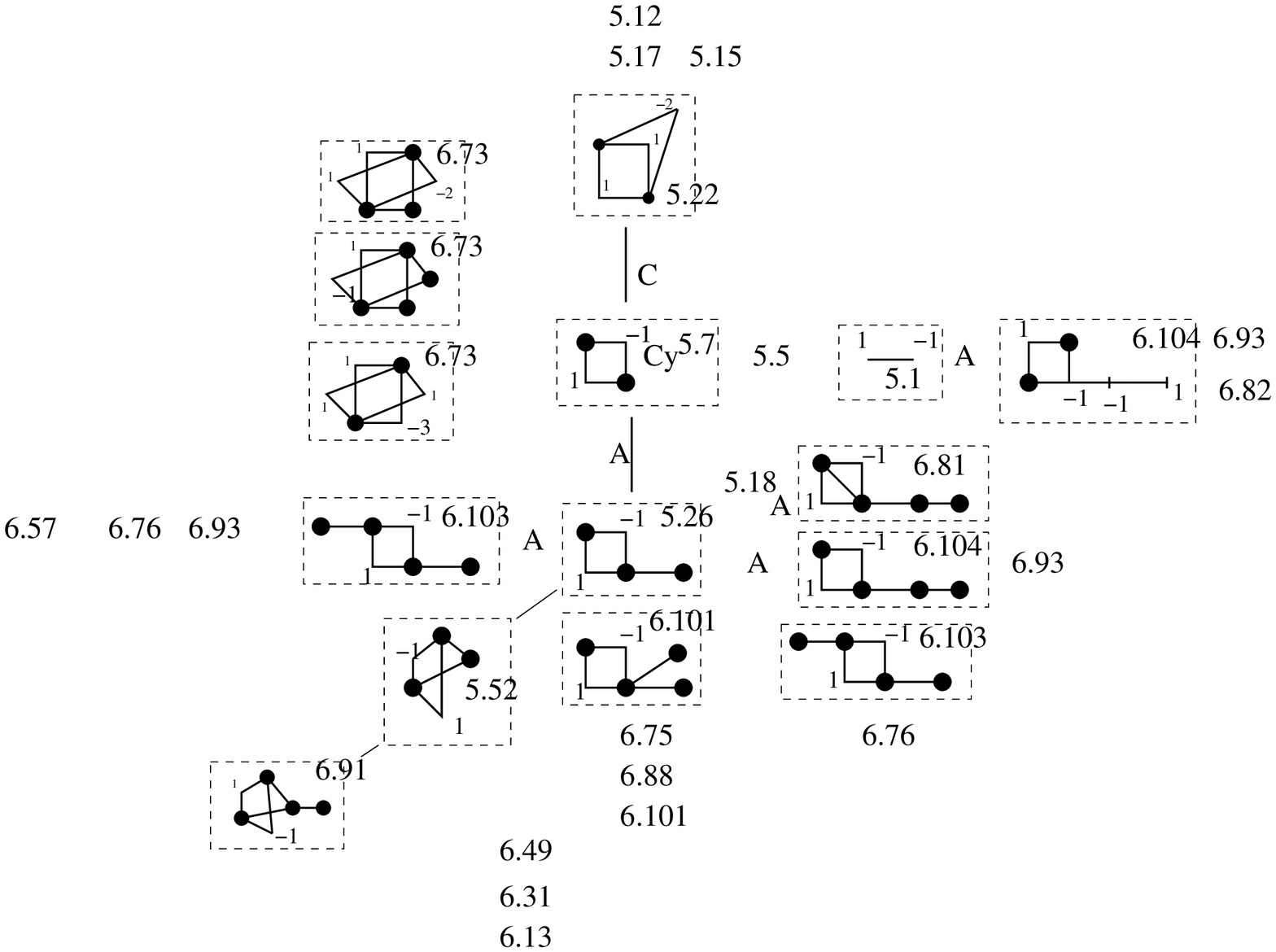,height=12 cm,width=16 cm,angle=0}
}
\caption{$2$-soft graphs} 
\label{l2}
\end{figure}

\subsection{$3$-soft graphs }

Fig. \ref{contrac3} shows a 3s graph generated by expansion
of graph 5.22.
\begin{figure} [H]
\centerline{
\epsfig{file=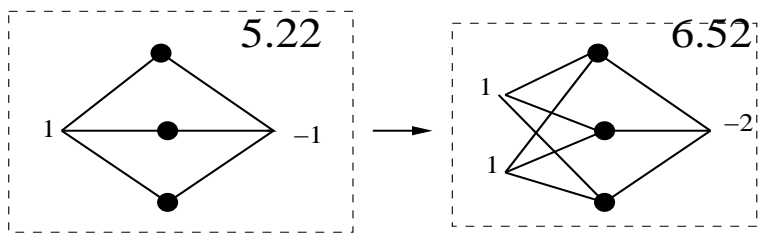,height=3 cm,width=10 cm,angle=0}
}
\caption{$3_s$ graphs: graphs generated by expansion. }
\label{contrac3}
\end{figure}
Fig. \ref{arti3} shows some 3s graphs generated by articulation on
graphs 5.2 and 5.22.
\begin{figure} [H]
\centerline{ \epsfig{file=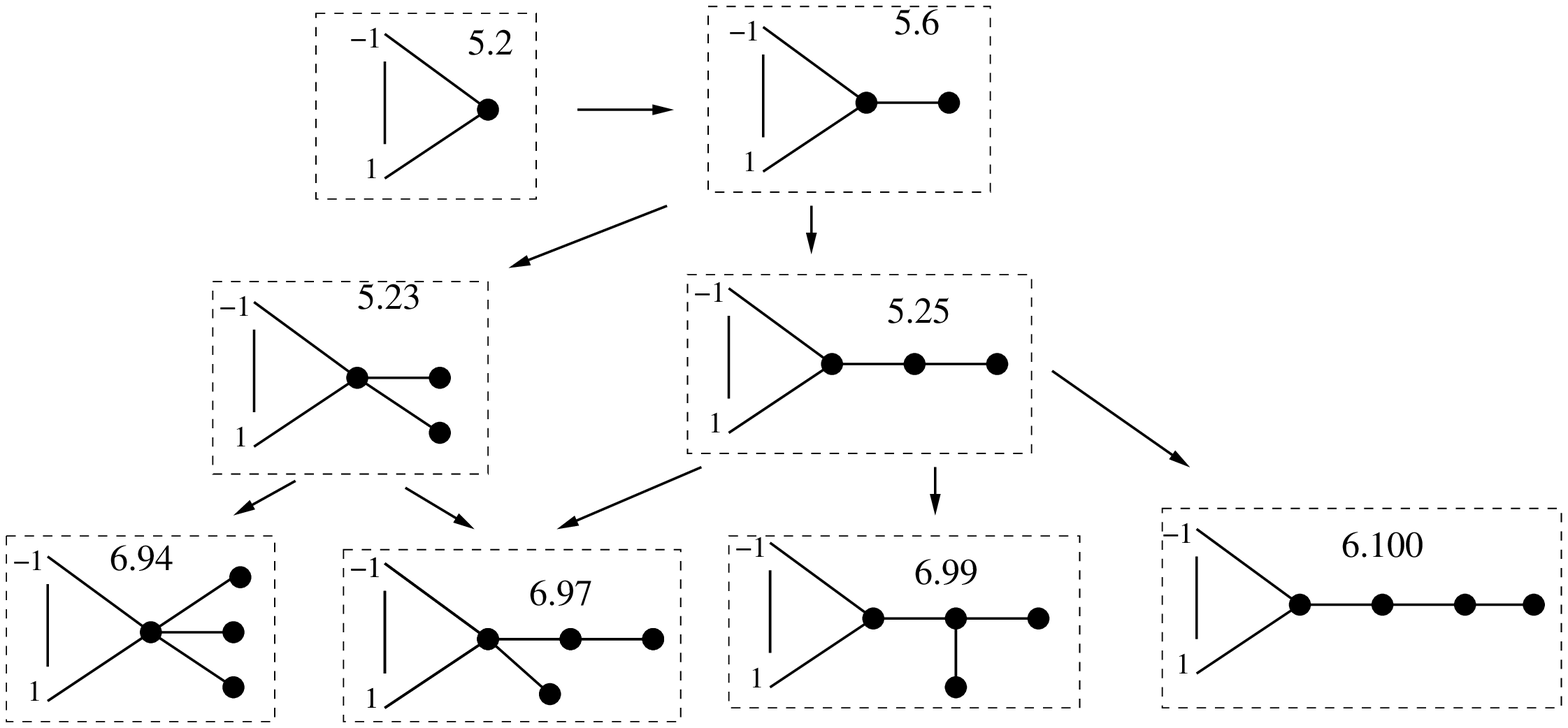,height=3 cm,width=14 cm,angle=0} }
\vspace{20pt}
\centerline{ \epsfig{file=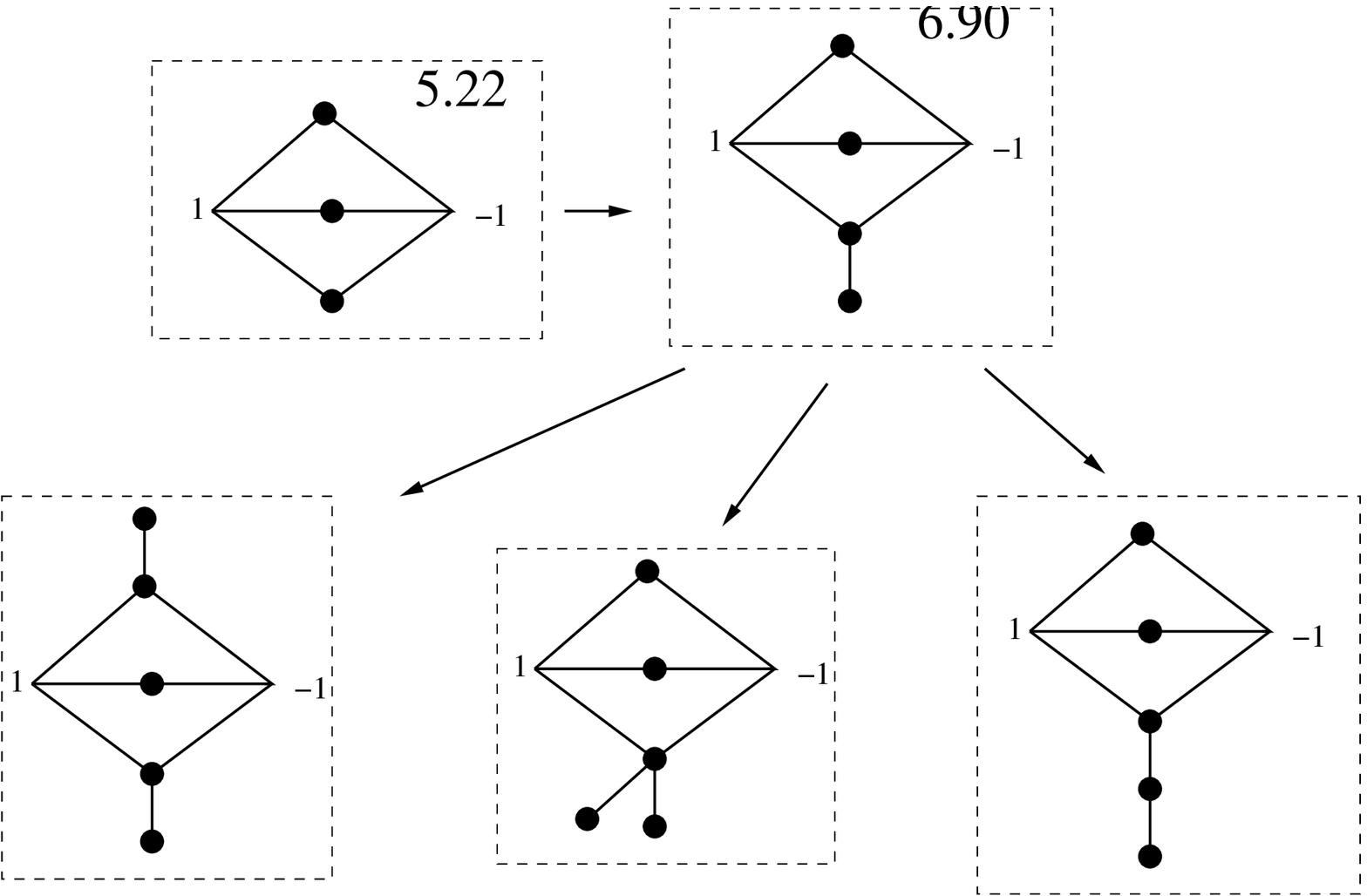,height=3 cm,width=14 cm,angle=0} }
\caption{$3_s$ graphs: graphs generated by articulation}
\label{arti3}
\end{figure}

\begin{figure} [H]
\centerline{
\epsfig{file=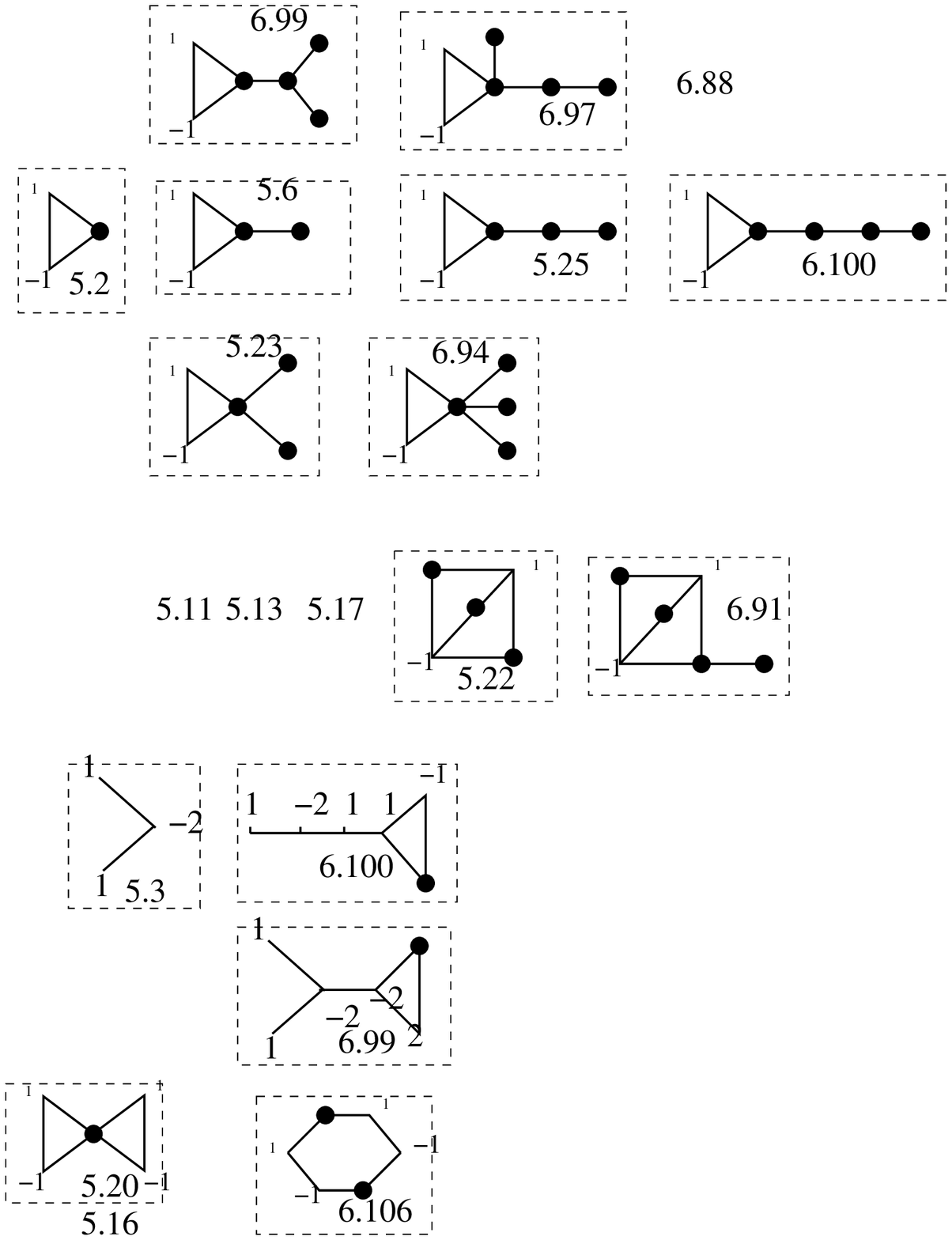,height=12 cm,width=16cm,angle=0}
}
\caption{$3$-soft graphs. }
\label{l3}
\end{figure}
Fig. \ref{l3} shows all 3s graphs with at most 6 vertices. Notice
how they are generated by graphs 5.2, 5.22 and 5.3. Graph 5.20
is the soldering of two graphs 5.2 .

\subsection{$4$-soft graphs }

Fig. \ref{arti4} shows some 4s graphs generated by articulation
on the graph 5.3.
\begin{figure} [H]
\centerline{ \epsfig{file=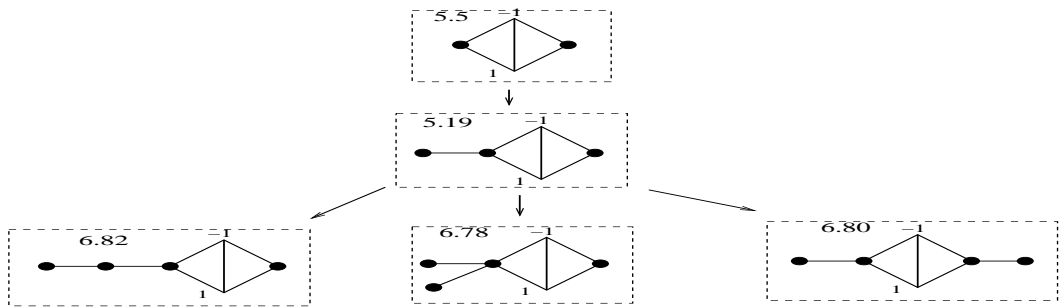,height=4 cm,width=14 cm,angle=0} }
\caption{$4_s$ graphs: graphs generated by articulation}
\label{arti4}
\end{figure}
Fig. \ref{l4} shows the 4s graphs with at most 6 vertices. Notice how
they are generated from graphs 5.5 (2 configurations) and 6.93. The
graph 5.7 is included to show its connection to 6.93 (replacing
a matching by a square).
\begin{figure} [H]
\centerline{
\epsfig{file=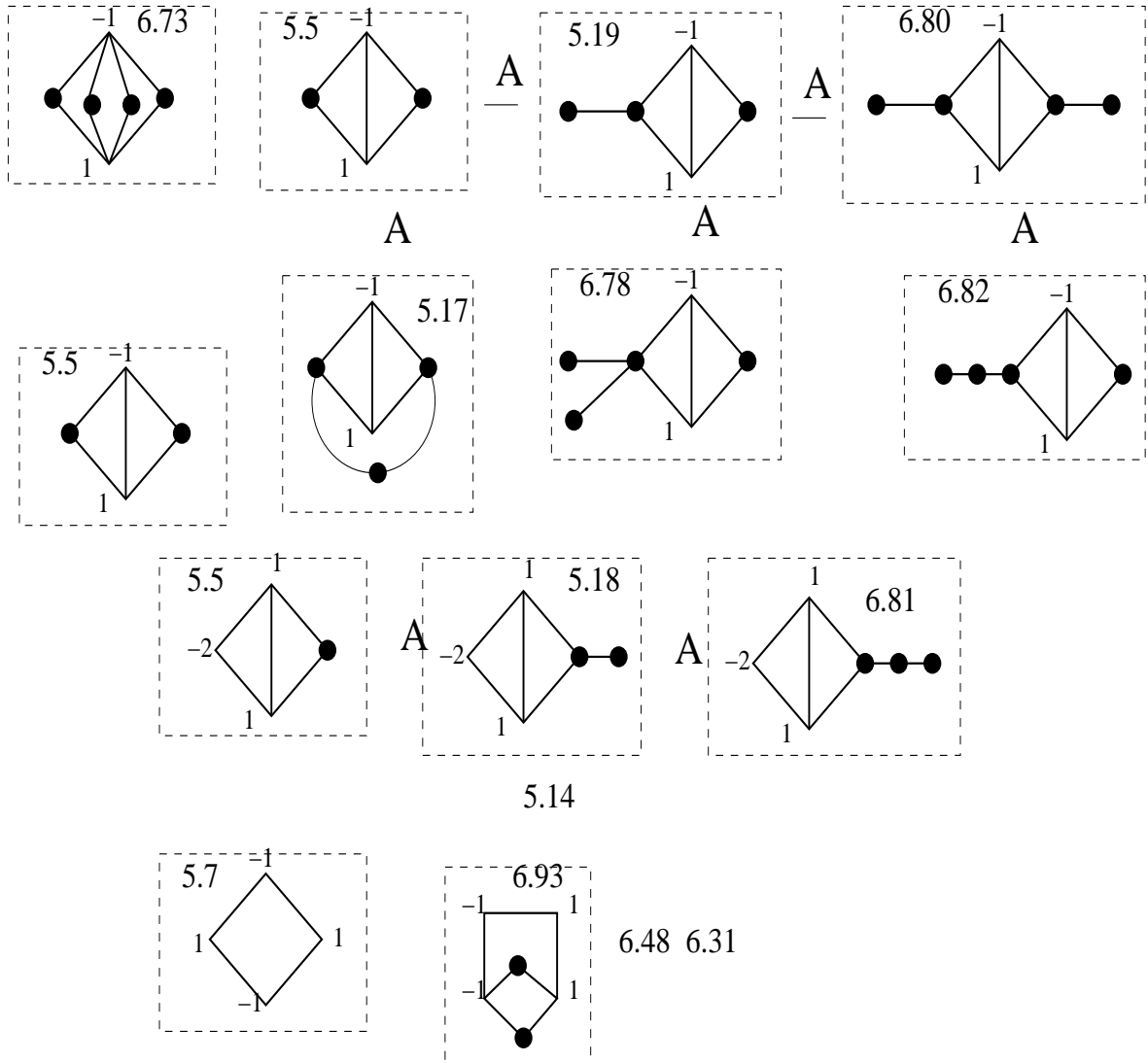,height=16 cm,width=18cm,angle=0}
}
\caption{$4$-soft graphs. 
}
\label{l4}
\end{figure}

\subsection{$5$-soft graphs }

Fig. \ref{l5} shows 5s graphs with at most 6 vertices. Notice
how they stem from graphs 6.70, 5.13 and two configurations 
of 5.15. 
\begin{figure} [H]
\centerline{
\epsfig{file=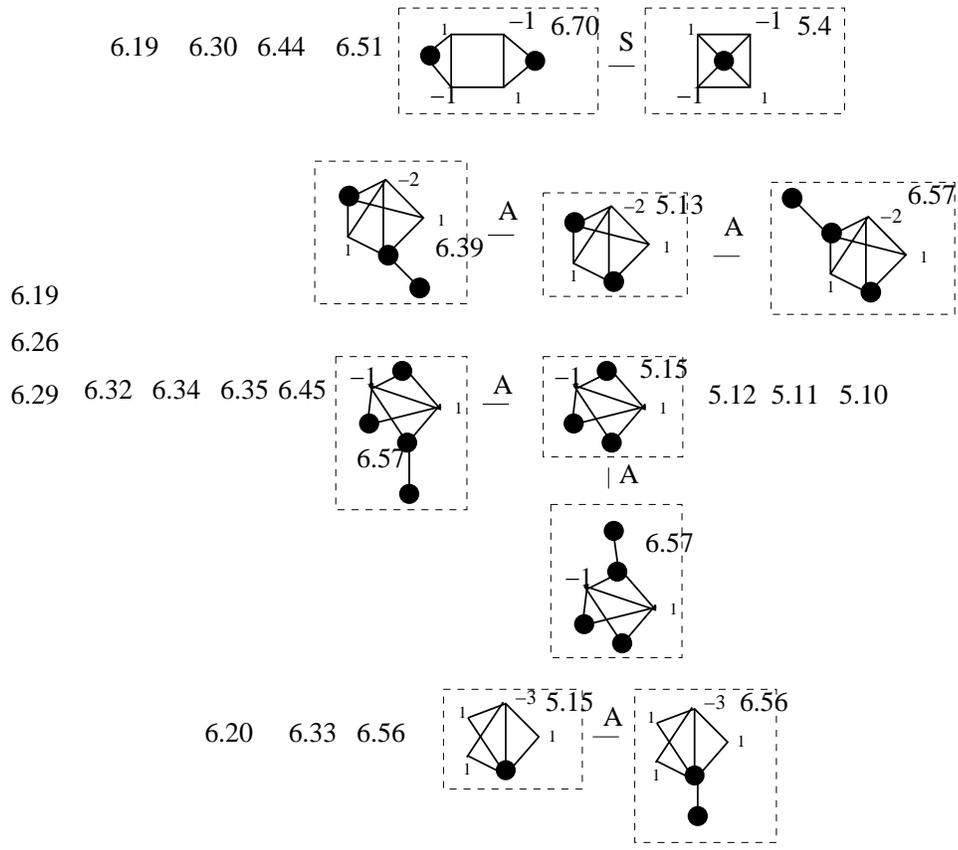,height=12 cm,width=14cm,angle=0}
}
\caption{$5$-soft graphs.
}
\label{l5}
\end{figure}

\subsection{$6$-soft graphs }

Fig. \ref{l6} shows 6s graphs with at most 6 vertices. Notice
how these graphs stem from graphs 6.9, 6.37, 6.2 (two configurations)
and 6.16.
\begin{figure} [H]
\centerline{
\epsfig{file=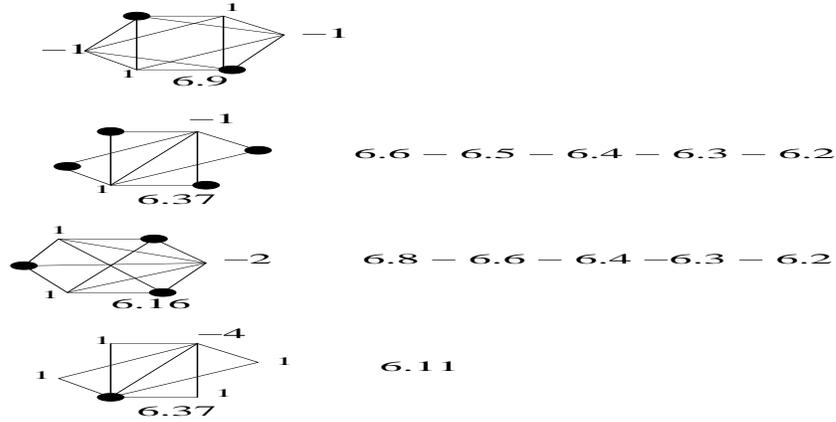,height=6 cm,width=12cm,angle=0}
}
\caption{$6$-soft graphs.
}
\label{l6}
\end{figure}

\subsection{x-soft graphs, x non integer }

As proven above, the only eigenvalues that
are non integer are irrational. For these, there can be
soft nodes. Among the 5 node graphs, we found irrational eigenvalues for the
chain 5 and the cycle 5. In addition, there are the following
\begin{table} [H]
\centering
\begin{tabular}{|c|c|c|}
\hline
nb. in          &  eigenvalue &  eigenvector \\  
classification  &             &             \\  \hline
5.16   & $\lambda_2=3 -\sqrt{2}$ & $(-0.27,-0.65,0,0.65,0.27)^T$ \\ \hline
5.16   & $\lambda_4=3 +\sqrt{2}$ & $(0.65,-0.27,0,0.27,-0.65)^T$ \\ \hline
5.21   & $\lambda_4=(7 +\sqrt{5} )/2$ & $(-0.6,0.6,0.37,0,-0.37)^T$ \\ \hline
5.21   & $\lambda_5=(7 -\sqrt{5} )/2$ & $(-0.37,0.37,-0.6,0,0.6)^T$ \\ \hline
5.24   & $\lambda_2=(5-\sqrt{13})/2 $ & $(-0.67,-0.2,0.2,0.67,0)^T$ \\ \hline
5.24   & $\lambda_5=(5+\sqrt{13})/2 $ & $(-0.2,0.67,-0.67,0.2,0)^T$ \\ \hline
5.30 (chain 5)   & $\lambda_4=(3 +\sqrt{5} )/2$ & $(-0.6,0.6,0.37,0,-0.37)^T$ \\ \hline
5.30 (chain 5)   & $\lambda_5=(3 -\sqrt{5} )/2$ & $(-0.37,0.37,-0.6,0,0.6)^T$ \\ \hline
\end{tabular}
\caption{\small\em Non trivial graphs with soft nodes and 
non integer eigenvalues.}
\label{tab7aa}
\end{table}

Remarks \\
The graph 5.16 is 3 soft. 
The graphs 5.21 and 5.24 are not part of an integer soft class. They are
\begin{figure} [H]
\centerline{
\epsfig{file=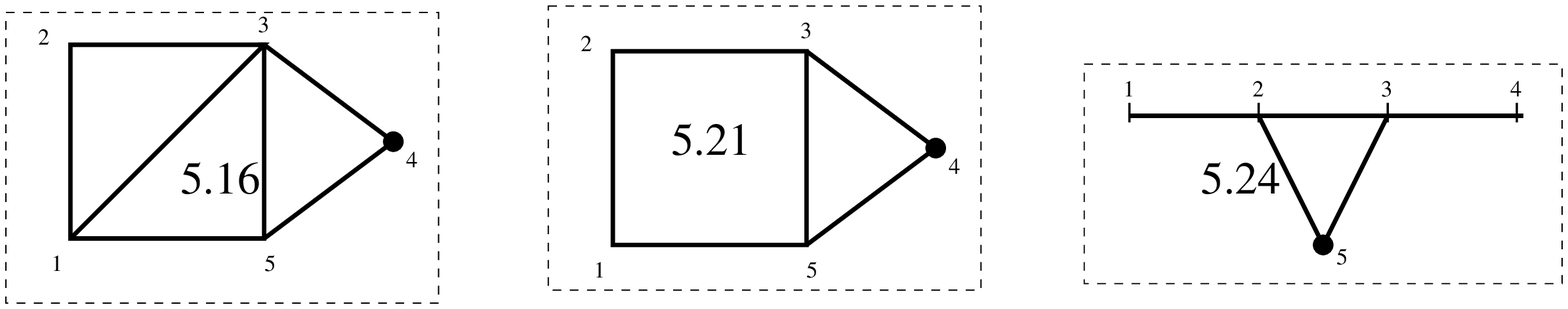,height=4 cm,width=12 cm,angle=0}
}
\caption{
The graphs 5.16, 5.21 and 5.24 with their soft node}
\label{f2124}
\end{figure}
\begin{itemize}
\item Graph 5.16 is a chain 4 with  a soft node added. 
\item Graph 5.21 is obtained from chain 5 (graph 5.30) by inserting a soft node.
\end{itemize}

\subsection{Minimal $\lambda$ soft graphs}

We computed the minimal $\lambda$ soft graphs for $\lambda=1, \dots, 6$.
These are presented in Fig. \ref{lmin}.
\begin{figure} [H]
\centerline{
\epsfig{file=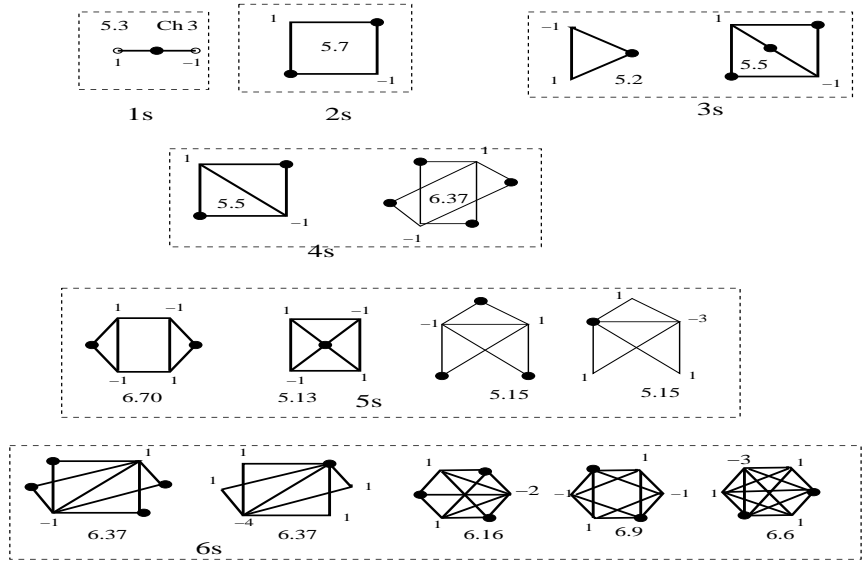,height=8 cm,width=12 cm,angle=0}
}
\caption{ The minimal $\lambda$ soft graphs for $\lambda=1,2,3,4,5$
and 6.}
\label{lmin}
\end{figure}
Note that there is a unique minimal $\lambda$-soft graph for
$\lambda=1$ and $2$. There are two minimal $3$-soft graphs
and $4$-soft graphs. There are four minimal $5$-soft graphs. 
The first two are generated by respectively inserting a soft node 
and adding a soft node to the minimal $4$-soft graph. The third 
and fourth ones
are obtained respectively by adding three soft nodes to the $2$ clique and 
adding a soft node to the $4$ star.

Three systematic ways to generate minimal $\lambda+1$-soft graphs are
(i) inserting a zero to a $\lambda$-soft graph, \\
(ii) adding a zero to a$\lambda$-soft graph and \\
(iii) adding a matching to a $\lambda-1$-soft graph.
One can therefore generate systematically minimal $7$-soft, $8$-soft..
graphs.

\section{Conclusion}

We reviewed families of graphs whose spectrum is known and presented
transformations that preserve an eigenvalue. 
The link, articulation and soldering were contained
in Merris \cite{merris} and we found two new transformations :
the regular expansion and the replacement of
a coupling by a square. We also showed transformations that
shift an eigenvalue :  insertion of a soft node (+1),
addition of a soft node (+1), insertion of a matching (+2).
The first is new and the second and third were found by Das \cite{das04}
and Merris \cite{merris} respectively.

From this appears a landscape of graphs formed by
families of $\lambda$-graphs 
connected by these transformations. These structures remain
to be understood. We presented the connections between
small graphs with up to six vertices. Is it possible to obtain
all the $\lambda$ graphs using a series of elementary transformations?
Or just part of these ?

We answered partially the question: can one predict
eigenvalues/eigenvectors from the geometry of a graph ? 
by examining the situation of a a $\lambda$
subgraph $G$ of a $\lambda$ graph $G"$. 
We showed that if the remainder graph $G'$ is $\lambda$, 
it is an articulation or a link of $G$. If not and if
$G$ and $G'$ share an eigenvector, the two may be related 
by adding one or several soft nodes to $G'$.

A number of the graphs we studied have irrational eigenvalues
and we can define $\lambda$ graphs for these as well because
the transformations apply. However we did not find any connection
between $\lambda$ graphs and $\mu$ graphs if 
$\lambda$ is an integer and $\mu$ an irrational.

\section{Appendix A: Graph classification }

The following tables indicate the graph classification we used. 
Each line in the "connections" column is the connection list of
the corresponding graph.

\begin{table} [H]
\centering
\begin{tabular}{|c|c|c|c|}
\hline
classification & nodes & links &  connections \\
\cite{crs01}   &       &       &     \\  \hline
1&2&1&12 \\
2&3&3&12~13~23~ \\
3&3&2&12~23 \\
4&4&6&12~13~14~23~24~34 \\
5&4&5&12~13~14~23~34 \\
6&4&4&12~13~23~34 \\
7&4&4&12~14~23~34~ \\
8&4&3&12~23~24 \\
9&4&3&12~23~34 \\
10&5&10&12~13~14~15~23~24~25~34~35~45 \\
11&5&9&12~13~14~15~23~24~34~35~45~ \\
12&5&8&12~14~15~23~42~25~34~45~ \\
13&5&8&12~13~15~23~24~34~35~45~ \\
14&5&7&12~13~14~23~24~34~35 \\
15&5&7&13~15~23~25~34~35~45~ \\
16&5&7&12~13~15~23~34~35~45~ \\
17&5&7&12~14~15~23~25~34~45~ \\
18&5&6&12~13~14~23~34~35~ \\
19&5&6&12~14~23~24~34~35~ \\
20&5&6&12~13~23~34~35~45~ \\
21&5&6&12~15~23~34~35~45~ \\
22&5&6&13~15~23~25~34~45~ \\
23&5&5&12~13~23~34~35 \\
24&5&5&12~23~25~35~34 \\
25&5&5&12~13~23~34~45 \\
26&5&5&12~14~23~34~35 \\
27&5&5&12~15~23~34~45~ \\
28&5&4&13~23~34~35 \\
29&5&4&12~13~14~45 \\
30&5&4&12~23~34~45 \\
\hline
\end{tabular}
\caption{\small\em Graphs of less than 5 nodes labelled $1$ to $30$ in classification \cite{crs01}.}
\label{c51a30}
\end{table}

\begin{table} [H]
\centering
\begin{tabular}{|c|c|c|c|}
\hline
classification & nodes & links &  connections \\
\cite{crs01}   &       &       &     \\  \hline
1&6&15&12~13~14~15~16~23~24~25~26~34~35~36~45~46~56 \\
2&6&14&12~13~15~16~23~24~25~26~34~35~36~45~46~56~ \\
3&6&13&12~14~15~16~23~24~25~26~34~35~45~46~56~ \\
4&6&13&12~13~15~16~23~24~25~26~34~35~45~46~56~ \\
5&6&12&12~13~15~16~23~25~26~34~35~36~45~56~ \\
6&6&12&12~13~14~15~16~23~25~34~35~36~45~56~ \\
7&6&12&12~13~15~16~23~24~25~26~34~35~45~56~ \\
8&6&12&12~13~15~16~23~24~34~35~36~45~46~56~ \\
9&6&12&12~13~15~16~23~24~26~34~35~45~46~56~ \\
10&6&11&12~13~14~15~23~24~25~34~35~45~56 \\ 
11&6&11&12~14~16~23~24~26~34~36~45~46~56 \\
12&6&11&12~13~15~16~23~25~26~34~35~45~56 \\
13&6&11&12~15~16~23~24~25~26~34~35~45~56 \\
14&6&11&12~13~15~16~23~25~26~34~36~45~56 \\
15&6&11&12~14~15~16~23~24~34~35~45~46~56 \\
16&6&11&12~14~15~16~23~25~34~35~36~45~56 \\
17&6&11&12~14~15~16~23~26~34~35~45~46~56 \\
18&6&11&12~15~16~23~24~26~34~35~45~46~56 \\
19&6&10&12~13~15~23~24~25~34~35~45~56 \\
20&6&10&12~13~14~15~23~24~34~35~45~56 \\
21&6&10&12~15~16~23~24~25~26~35~45~56 \\
22&6&10&12~13~15~16~23~34~35~36~45~56 \\
23&6&10&12~16~23~25~26~34~35~36~45~56 \\
24&6&10&12~15~16~23~25~26~34~35~45~56 \\
25&6&10&12~13~14~15~16~23~34~36~45~56 \\
26&6&10&12~14~16~23~34~35~36~45~46~56 \\
27&6&10&12~15~16~23~26~34~35~36~45~56 \\
28&6&10&12~14~16~23~24~34~35~45~46~56 \\
29&6&10&12~14~16~23~24~26~34~36~45~56 \\
30&6&10&12~15~61~23~24~25~34~36~45~56 \\
\hline
\end{tabular}
\caption{\small\em 6 node graphs labelled $1$ to $30$ in classification \cite{crs01}.}
\label{c1a30}
\end{table}

\begin{table} [H]
\centering
\begin{tabular}{|c|c|c|c|}
\hline
classification & nodes & links &  connections \\
\cite{crs01}   &       &       &     \\  \hline
31&6&10&12~15~16~23~24~26~34~35~45~56 \\
32&6&10&12~14~16~23~25~26~34~36~45~56 \\
33&6&9&12~15~23~24~25~34~35~45~56 \\
34&6&9&12~14~15~23~24~25~34~45~56 \\
35&6&9&12~13~14~15~23~24~34~45~56 \\
36&6&9&12~13~14~23~24~34~45~46~56 \\
37&6&9&12~13~14~15~16~24~34~45~46 \\
38&6&9&12~14~15~23~25~34~35~45~56 \\
39&6&9&12~13~15~23~24~25~34~45~56 \\
40&6&9&12~13~16~23~34~35~36~46~56 \\
41&6&9&12~16~23~34~35~36~45~46~56~ \\
42&6&9&12~16~23~24~26~34~45~46~56~ \\
43&6&9&12~13~16~23~34~35~36~45~56~ \\
44&6&9&12~13~16~23~34~36~45~46~56~ \\
45&6&9&12~16~23~25~34~35~36~45~56~ \\
46&6&9&12~13~15~16~23~34~36~45~56~ \\
47&6&9&12~13~16~23~25~26~34~45~56~ \\
48&6&9&12~15~16~23~26~34~35~45~56~ \\
49&6&9&12~15~16~23~24~26~35~45~56~ \\
50&6&9&12~14~15~16~23~34~36~45~56~ \\
51&6&9&12~15~24~16~23~34~36~45~56~ \\
52&6&9&12~14~16~23~25~34~36~45~56~ \\
53&6&8&12~13~14~23~24~34~45~46~ \\
54&6&8&12~13~14~23~24~25~34~36~ \\
55&6&8&12~13~14~23~24~34~45~56~ \\
56&6&8&13~15~23~25~34~35~45~56~ \\
57&6&8&12~14~23~24~25~34~45~56~ \\
58&6&8&12~15~23~25~34~35~45~56~ \\
59&6&8&12~13~15~23~34~35~45~56~ \\
60&6&8&12~14~15~23~24~34~45~56~ \\
\hline
\end{tabular}
\caption{\small\em 6 node graphs labelled $31$ to $60$ in classification \cite{crs01}. }
\label{c31a60}
\end{table}

\begin{table} [H]
\centering
\begin{tabular}{|c|c|c|c|}
\hline
classification & nodes & links &  connections \\
\cite{crs01}   &       &       &     \\  \hline
61&6&8&12~13~14~15~16~23~45~46~ \\
62&6&8&12~14~23~42~34~35~36~56~ \\
63&6&8&12~14~15~23~25~34~45~56~ \\
64&6&8&12~13~15~23~25~34~45~56~ \\
65&6&8&12~13~15~23~24~34~45~56~ \\
66&6&8&12~16~24~34~36~45~56~46 \\
67&6&8&12~13~16~23~34~36~45~56~ \\
68&6&8&12~13~16~23~34~35~45~56~ \\
69&6&8&12~15~16~23~26~34~45~56~ \\
70&6&8&12~13~16~23~34~45~46~56~ \\
71&6&8&12~13~16~23~34~35~46~56~ \\
72&6&8&12~15~16~23~34~36~45~56~ \\
73&6&8&12~15~23~24~26~35~45~56 \\
74&6&8&12~14~16~23~25~34~45~56~ \\
75&6&7&12~13~14~23~34~35~36 \\
76&6&7&12~23~24~25~34~45~46 \\
77&6&7&12~14~23~24~25~34~36 \\
78&6&7&12~14~23~24~34~35~36 \\
79&6&7&12~13~23~34~36~35~45~ \\
80&6&7&12~23~25~34~35~45~46~ \\
81&6&7&12~13~14~23~34~35~56 \\
82&6&7&12~14~23~24~34~35~56 \\
83&6&7&12~13~23~34~35~45~46 \\
84&6&7&12~13~23~34~45~46~56~ \\
85&6&7&12~15~23~24~34~45~46 \\
86&6&7&12~13~15~23~34~45~46 \\
87&6&7&12~13~15~23~34~45~46 \\
87B&6&7&12~13~15~24~34~45~56~ \\
88&6&7&12~14~23~34~35~36~56~ \\
89&6&7&12~15~16~23~34~45~56~ \\
90&6&7&13~15~23~25~34~45~56~ \\
\hline
\end{tabular}
\caption{\small\em 6 node graphs labelled $61$ to $90$ in classification \cite{crs01}. Note that 87B is absent from \cite{crs01}.}
\label{c61a90}
\end{table}

\begin{table} [H]
\centering
\begin{tabular}{|c|c|c|c|}
\hline
classification & nodes & links &  connections \\
\cite{crs01}   &       &       &     \\  \hline
91&6&7&12~14~23~25~34~45~56 \\
92&6&7&12~16~23~34~36~45~56~ \\
93&6&7&12~15~23~34~36~45~56~ \\
94&6&6&12~13~23~34~35~36 \\
95&6&6&13~23~34~35~45~56 \\
96&6&6&12~23~25~34~35~56 \\
97&6&6&12~13~23~34~35~56 \\
98&6&6&12~23~35~34~45~56 \\
99&6&6&12~23~24~45~46~56 \\
100&6&6&12~23~34~45~46~56\\
101&6&6&12~14~23~34~35~36 \\
102&6&6&12~14~23~25~34~36 \\
103&6&6&12~23~42~35~45~56 \\
103&6&6&12~14~23~34~35~56 \\
105&6&6&12~15~23~34~45~46 \\
106&6&6&12~16~23~34~45~56~ \\
107&6&5&16~26~36~46~56 \\
108&6&5&14~24~34~45~56 \\
109&6&5&13~23~34~45~46~ \\
110&6&5&12~23~34~36~45~ \\
111&6&5&12~23~34~45~46~ \\
112&6&5&12~23~34~45~56 \\
\hline
\end{tabular}
\caption{\small\em 6 node graphs labelled $91$ to $112$ in classification \cite{crs01}.}
\label{c91a112}
\end{table}

\section{Appendix B: sets $1_s,~2_s,~3_s,~4_s$ and $5_s$}

We give here the tables for the sets 1s, 2s, 3s, 4s and 5s
for 5 node graph and 6 node graphs. The numbering of the graphs follow
the ones given by Cvetkovic \cite{crs01} for 5 and less nodes and 6 nodes graphs
respectively. 

\subsection{1s}

\begin{table} [H]
\centering
\begin{tabular}{|c|c|c|c|c|}
\hline
nodes & links & classification \cite{crs01}  & eigenvector &  connection \\
      &       &                              &              &  \\  \hline
3  & 2  & 3    &  $(-1, 0, 1)$     &        \\ \hline
4  & 3  &  8   &  $(0, 0, -1, 1)$  &  \\ \hline
4  & 4  &  6   &  $(1, 1, 0, -2)$  & expansion on 5.3  \\ \hline
5  & 4  & 28   &  $(1,1,0,-1,-1)$  & \\ 
5  & 4  & 28   &  $(1,0,0,0,-1)$   & articulation on 5.3 \\ 
5  & 4  & 28   &  $(1,1,1,0, -3)$  &star 4\\ 
5  & 4  & 29   &  $(0,1,-1,0, 0)$  &  articulation on 5.3  \\ \hline
5  & 5  & 23   &  $(0,0,0,1,-1)$   & articulation on 5.3\\ 
5  & 5  & 23   &  $(1,1,0,-2,0)$   & expansion on 5.3\\ \hline
5  & 6  & 18   &  $(1,1,0,1, -3)$  & \\ 
5  & 6  & 20   &  $(1,1,0,-1,-1)$  & articulation on 5.28 \\ \hline
5  & 7  & 14   &  $(1,1,0,1, -3)$  & \\ \hline
\end{tabular}
\caption{\small\em Five node graphs with soft nodes and eigenvalue $1$.}
\label{tab3}
\end{table}

\begin{table} [H]
\centering
\begin{tabular}{|c|c|c|c|c|}
\hline
nodes & links & classification \cite{crs01}  & eigenvector &  connection \\
      &       &                              &              &  \\  \hline
6  & 11 &  10    &  $(1,1,1,1,0,-4)$ & link on 19\\ \hline
6  & 10 &  19    &  $(1,1,1,1,0,-4)$ & link on 33\\ \hline
6  & 9 &  33    &  $(1,1,1,1,0,-4)$ & link on 38\\ 
6  & 9 &  36    &  $(2,2,2,0,-3,-3)$ & link on 61\\ 
6  & 9 &  38    &  $(1,1,1,1,0,-4)$ & link 58\\ 
6  & 9 &  53    &  $(0,0,0,0,1,-1)$ & link 75 \\ \hline
6  & 9 &  53    &  $(1,1,1,0,-3,0)$ & link 75 \\ \hline
6  & 8 &  56    &  $(1,1,1,1,0,-4)$ & expansion on 5.3\\ 
6  & 8 &  58    &  $(1,1,1,1,0,-4)$ & expansion on 5.3\\ 
6  & 8 &  61    &  $(2,2,2,0,-3,-3)$ & link on 94 \\ \hline
6  & 7 &  75    &  $(0,0,0,0,1,-1)$ & articulation on 5.3\\ 
6  & 7 &  75    &  $(1,1,0,1,-3,0)$ & link on 101 \\ 
6  & 7 &  78    &  $(0,0,0,0,1,-1)$ & link on 101 \\ 
6  & 7 &  79    &  $(1,1,0,0,0,-2)$ & expansion on 5.3\\ 
6  & 7 &  79    &  $(1,1,0,-1,-1,0)$&  link on 94\\ 
6  & 7 &  92    &  $(1,1,0,-1,-1,0)$&  link on 106\\ \hline
6  & 6 &  94    &  $(1,1,0,-2,0,0)$ & link on 107\\ 
6  & 6 &  94    &  $(3,3,0,-2,-2,-2)$&  link on 107\\ 
6  & 6 &  94    &  $(1,-1,0,0,0,0)$ & link on 95\\ 
6  & 6 &  95    &  $(1,-1,0,0,0,0)$ & articulation on 5.3 \\ 
6  & 6 &  97    &  $(1,1,0,-2,0,0)$ & link on 108 \\ 
6  & 6 &  99    &  $(1,0,-1,0,0,0)$ & articulation on 5.3\\ 
6  & 6 &  101   &  $(0,0,0,0,1,-1)$ & articulation on 5.3\\ 
6  & 6 & 106   &   $(0,1,1,0,-1,-1)$& link between two 5.3 \\ \hline
6  & 5 &  107   &  $(-1,1,0,0,0,0)$ & articulation on 5.3\\ 
6  & 5 &  107   &  $(1,1,-1,-1,0,0)$&  soldering two 5.3 and articulation\\ 
6  & 5 &  107   &  $(1,1,0,-2,0,0)$&  articulation on 5.3\\ 
6  & 5 &  108   &  $(-1,1,0,0,0,0)$ & articulation on 5.3\\ 
6  & 5 &  108   &  $(-1,0,1,0,0,0)$ & articulation on 5.3\\ 
6  & 5 &  109   &  $(-1,1,0,0,0,0)$ & articulation on 5.3 \\ 
6  & 5 &  109   &  $(0,0,0,0,-1,1)$ & articulation on 5.3 \\ 
6  & 5 &  111   &  $(0,0,0,0,-1,1)$ & articulation on 5.3 \\ \hline
\end{tabular}
\caption{\small\em Six node graphs with soft nodes and eigenvalue $1$.}
\label{tab3a}
\end{table}

\subsection{2s}

\begin{table} [H]
\centering
\begin{tabular}{|c|c|c|c|c|}
\hline
nodes & links & classification \cite{crs01}  & eigenvector &  connection \\
      &       &                              &              &  \\  \hline
4  & 5  & 5   &    $(1,0,-1,0)$ &   link on 5.7     \\ \hline
4  & 4  & 7   &    $(1,0,-1,0)$ &        \\ 
4  & 4  & 7   &    $(0,1,0,-1)$ &        \\ \hline
5  & 8  &  12   &   $(1, 0, -2,0, 1)$ &  link on 5.17      \\ \hline
5  & 7  & 15   &   $(1,0,-1,0,0)$ &  articulation on 5.7      \\ 
5  & 7  &  15   &   $(1, 0  1, 0, -2)$&  link on 5.17        \\ 
5  & 7  & 17   &   $(1,0,-2,0, 1)$ &        \\ \hline
5  & 6  & 18   &   $(0,1,0,-1,0)$ &  articulation 5.7      \\ 
5  & 6  &  22   &   $(0,1, 0, -1,0)$ & add a zero to 5.3 and articulation       \\ 
5  & 6  &  22   &   $(1, 0, 0,-1,0)$ & add a zero to 5.3 and articulation       \\ \hline
5  & 5  & 26   &   $(0,1,0,-1,0)$ & articulation 5.7       \\ \hline
\end{tabular}
\caption{\small\em Five node graphs with soft nodes and eigenvalue $2$.}
\label{tab4}
\end{table}

\begin{table} [H]
\centering
\begin{tabular}{|c|c|c|c|c|}
\hline
nodes & links & classification \cite{crs01}  & eigenvector &  connection \\
      &       &                              &              &  \\  \hline
6  & 12 &  5   &  $(1,1,0,-3,0,1)$ &      \\ \hline
6  & 11 &  11  &  $(1,1,1,0,-3,0)$ &      \\ 
6  & 11 &  13  &  $(1,0,-1,-1,0,1)$ &      \\ 
6  & 11 &  14  &  $(1,1,0,-3,0,1)$ &      \\ \hline
6  & 10 &  21  &  $(1,0,-1,-1,0,1)$ &      \\ 
6  & 10 &  21  &  $(0,0,1,-1,0,0)$ &      \\ 
6  & 10 &  29  &  $(1,1,1,0,-3,0)$ &      \\ 
6  & 10 &  31  &  $(1,0,-1,-1,0,1)$ & link on 37 \\ \hline
6  & 9 &  33  &  $(-2,0,1,1,0,0)$ &  link on 56    \\ 
6  & 9 &  37  &  $(1,0,0,0,0,-1)$ &  link on 73 \\
6  & 9 &  37  &  $(0,0,0,-1,0,1)$ & link on 73 \\
6  & 9 &  37  &  $(1,0,1,-1,0,-1)$ &  link on 73 \\
6  & 9 &  40  &  $(0,0,0,0,1,-1)$ &  addition of a 0 to 5.3, articulation    \\ 
6  & 9 &  49  &  $(0,0,1,-1,0,0)$ &  addition of a 0 to 5.3, articulation    \\ 
6  & 9 &  49  &  $(1,0,-1,-1,0,1)$ & link on 69     \\  \hline
6  & 8 &  56  &  $(-1,1,0,0,0,0)$ &  addition of a 0 to 5.3, articulation    \\ 
6  & 8 &  56  &  $(1,1,0,-2,0,0)$ &  link on 64    \\ 
6  & 8 &  57  &  $(-1,0,1,0,0,0)$ & link 93 \\ 
6  & 8 &  61  &  $(0,0,0,0,-1,1)$ &  addition of a 0 to 5.3, articulation    \\ 
6  & 8 &  64  &  $(1,1,0,-2,0,0)$ &  expansion of 5.7    \\ 
6  & 8 &  66  &  $(0,0,1,0,-1,0)$ &  addition of a 0 to 5.3    \\ 
6  & 8 &  69  &  $(0,1,1,-1,-1,0)$ &  link 5.7 and 5.3    \\ 
6  & 8 &  71  &  $(0,0,0,1,-1,0)$ & addition of a 0 to 5.3 \\ 
6  & 8 &  73  &  $(1,0,0,0,0,-1)$ &  addition of a 0 to 5.3    \\ 
6  & 8 &  73  &  $(0,0,0,-1,0,1)$ & addition of a 0 to 5.3     \\ 
6  & 8 &  73  &  $(1,0,1,-1,0,-1)$ &  soldering two 5.7    \\ 
6  & 8 &  74  &  $(0,-1,0,1,0,0)$ &  link and articulation 5.7    \\ \hline
6  & 7 &  75  &  $(0,-1,0,1,0,0)$ &  link 101 \\ 
6  & 7 &  76  &  $(0,0,-1,0,1,0)$ &  link 103 \\ 
6  & 7 &  81  &  $(0,-1,0,1,0,0)$ &  link and articulation 5.7    \\ 
6  & 7 &  82  &  $(1,0,-1,0,-1,1)$ & link 104 \\ 
6  & 7 &  88  &  $(0,-1,0,1,0,0)$ &  articulation on 5.7    \\ 
6  & 7 & 90  &   $(1,0,0,-1,0,0)$ &  articulation on 5.7    \\ 
6  & 7 &  90  &  $(1,-1,0,0,0,0)$ &   articulation on 5.7   \\ 
6  & 7 &  91  &  $(1,0,-1,0,0,0)$ &  addition of a 0 to 5.3    \\ 
6  & 7 &  92  &  $(-1,1,1,1,-1,-1)$ & link on 5.1     \\ 
6  & 7 &  93  &  $(0,0,0,1,0,-1)$ & addition of a 0 to 5.3, articulation     \\ 
6  & 7 &  93  &  $(1,-1,-1,0,1,0)$ &      \\ \hline
6  & 6 &  101 &  $(0,1,0,-1,0,0)$ &  addition of a 0 to 5.3, articulation \\ 
6  & 6 &  103 &  $(0,0,1,-1,0,0)$ & addition of a 0 to 5.3, articulation \\ 
6  & 6 &  104 &  $(0,1,0,-1,0,0)$ &  addition of a 0 to 5.3, articulation \\ 
6  & 6 &  104 &  $(1,0,-1,0,-1,1)$ & articulation on 5.7     \\ \hline
\end{tabular}
\caption{\small\em Six node graphs with soft nodes and eigenvalue $2$.}
\label{tab3b}
\end{table}

\subsection{3s}

\begin{table} [H]
\centering
\begin{tabular}{|c|c|c|c|c|}
\hline
nodes & links & classification \cite{crs01}  & eigenvector &  connection\\
      &       &                              &       & \\  \hline
3     & 3     & 2                            & $(-1,1,0)$  &  \\ \hline
4  & 4  & 6   &        $(-1,1,0,0)$     & articulation 5.3\\ \hline
5  & 7  & 11   &        $(0,-1,0,0,1)$  & articulation 5.3\\ \hline
5  & 6  & 13   &        $(-1,0,0,-1,0)$ & articulation 5.3 \\ \hline
5  & 8  & 13   &        $(0,1,0,0,-1)$  & articulation 5.3\\ \hline
5  & 7  & 16   &        $(-1,1,0,-1,1)$ & addition of \\ 
   &    &      &                        & zero to chain 4  \\ \hline
5  & 7  &  17   &      $(0,1, 0, -1,0)$ & articulation 5.3 \\ \hline
5  & 6  & 20   &        $(-1,1,0,0,0)$  & articulation 5.3 \\ \hline
5  & 6  & 20   &        $(0,0,0,-1,1)$  & articulation 5.3 \\ \hline
5  & 6  &  22   &      $(0, 0, -1,0,1)$ & articulation 5.3 \\ \hline
5  & 5  & 23   &        $(-1,1,0,0,0)$  & articulation 5.3\\ \hline
5  & 5  & 25   &        $(-1,1,0,0,0)$  & articulation 5.3 \\ \hline
\end{tabular}
\caption{\small\em Five node graphs with soft nodes and eigenvalue $3$.}
\label{tab5}
\end{table}

\begin{table} [H]
\centering
\begin{tabular}{|c|c|c|c|c|}
\hline
nodes & links & classification \cite{crs01}  & eigenvector & connection\\
      &       &                              &             & \\  \hline
6  & 13 & 3    &   $(-1,0,2,0,0,-1)$  & link 8 \\ \hline
6  & 12 & 6    &   $(0,-1,0,1,0,0)$   & link 11\\ \hline
6  & 12 & 6    &   $(0,-1,0,-1,0,2)$  & link 8 \\ 
6  & 12 & 8    &   $(0,-2,0,0,1,1)$   & link 11 \\ \hline
6  & 11 & 11   &   $(1,0,-1,0,0,0)$   & link 16 \\ 
6  & 11 & 16   &   $(0,1,0,-1,0,0)$   & link 25\\ 
6  & 11 & 16   &   $(0,-1,0,-1,0,2)$  & link 17 \\ 
6  & 11 & 17   &   $(0,2,0,-1,-1,0)$  & link 52\\ 
6  & 11 & 17   &   $(1,0,-2,0,0,1)$   & link 52\\ \hline
6  & 10 & 19   &   $(1,0,0,-1,0,0)$   & link 39 \\ 
6  & 10 & 25   &   $(0,0,0,-1,0,1)$   & \\ 
6  & 10 & 29   &   $(1,0,-1,0,0,0)$   & \\ 
6  & 10 & 32   &   $(0,1,0,-1,0,0)$   & link 52  \\ 
6  & 10 & 32   &   $(1,0,-1,0,0,0)$   & link 52  \\ \hline
6  & 9  & 36   &   $(0,0,0,0,1,-1)$   & articulation 5.2\\ 
6  & 9  & 38   &   $(1,0,-1,0,0,0)$   & articulation 5.13\\ 
6  & 9  & 38   &   $(0,1,0,-1,0,0)$   & link 63 \\ 
6  & 9  & 39   &   $(1,0,0,-1,0,0)$   & link 63  \\ 
6  & 9  & 51   &   $(1,1,0,-1,-1,0)$  & link 106 \\ 
6  & 9  & 51   &   $(0,0,1,-1,1,-1)$  & link 106 \\ 
6  & 9  & 52   &   $(0,1,0,1,0,-2)$   &   \\ 
6  & 9  & 52   &   $(1,0,1,0,-2,0)$   &    \\ 
6  & 9  & 52   &   $(1,0,-1,0,0,0)$   & link 106 \\ 
6  & 9  & 52   &   $(0,1,0,-1,0,0)$   & link 106 \\ \hline
6  & 8  & 58   &   $(-1,1,1,-1,0,0)$  & link 79   \\ 
6  & 8  & 61   &   $(0,-1,1,0,0,0)$   & articulation 5.2 \\ 
6  & 8  & 62   &   $(0,0,0,0,-1,1)$   & articulation 5.2 \\ 
6  & 8  & 63   &   $(0,1,0,-1,0,0)$   & link 91 \\ 
6  & 8  & 70   &   $(0,-1,1,1,-1,0)$  & link 106\\ 
6  & 8  & 70   &   $(-1,0,1,1,0,-1)$  & link 106\\ 
6  & 8  & 74   &   $(-1,0,1,-1,0,1)$  & link 106\\ 
6  & 8  & 74   &   $(0,0,0,-1,0,1)$   & link 106\\ 
6  & 8  & 74   &   $(-1,0,1,0,0,0)$   & link 106\\ \hline
6  & 7  & 79   &   $(-1,1,0,0,0,0)$  & articulation 5.2 \\ 
6  & 7  & 79   &   $(0,0,0,-1,1,0)$  & articulation 5.2 \\ 
6  & 7  & 83   &   $(-1,1,0,0,0,0)$   & articulation 5.2 \\ 
6  & 7  & 84   &   $(-1,0,1,0,0,0)$   & articulation 5.2 \\ 
6  & 7  & 84   &   $(0,0,0,-1,0,1)$   & articulation 5.2 \\ 
6  & 7  & 88   &   $(0,0,0,0,-1,1)$   & articulation 5.2\\ 
6  & 7  & 91   &   $(0,-1,0,1,0,0)$   &  \\ 
6  & 7  & 92   &   $(1,-1,0,1,-1,0)$  & link 106 \\ 
6  & 7  & 92   &   $(0,1,-1,0,1,-1)$  & link 106 \\ \hline
6  & 6  & 94   &   $(0,0,0,0,1,-1)$   & articulation 5.2\\ 
6  & 6  & 97   &   $(0,0,0,0,1,-1)$   & articulation 5.2\\ 
6  & 6  & 99   &   $(1,-2,1,2,-2,0)$  & soldering P3 and C3 \\ 
6  & 6  & 99   &   $(0,0,0,0,1,-1)$   & articulation 5.2\\ 
6  & 6  & 100  &   $(1,-2,1,1,-1,0)$  & soldering P3 and C3 \\ 
6  & 6  & 100  &   $(0,0,0,0,1,-1)$   & articulation 5.2\\ 
6  & 6  & 106  &   $(-1,0,1,-1,0,1)$  & cycle 6  \\ 
6  & 6  & 106  &   $(-1,1,0,-1,1,0)$  & cycle 6  \\ \hline
\end{tabular}
\caption{\small\em Six node graphs with soft nodes and eigenvalue $3$.}
\label{tab8}
\end{table}

\subsection{4s}

\begin{table} [H]
\centering
\begin{tabular}{|c|c|c|c|c|}
\hline
nodes & links & classification \cite{crs01}  & eigenvector &  connection \\
      &       &                              &             &   \\  \hline
4     & 5     & 5              &  $(1,-2,1,0)$ &     \\ \hline
5     & 8     & 12             &  $(1,0,0,0,-1)$  &  link on 17   \\ 
5     & 7     & 14             &  $(0,1,0,-1,0)$  &     \\ 
5     & 7     & 14             &  $(-1,0,0,1,0)$  & link on 19  \\ 
5     & 7     & 17             &  $(-1,0,0,0,1)$  &     \\ 
5     & 7     & 18             &  $(-2,1,0,1,0)$  & articulation on 5    \\ 
5     & 7     & 19             &  $(0,1,0,-1,0)$  &     \\  \hline
\end{tabular}
\caption{\small\em Five node graphs with soft nodes and eigenvalue $4$.}
\label{tab6b}
\end{table}

\begin{table} [H]
\centering
\begin{tabular}{|c|c|c|c|c|}
\hline
nodes & links & classification \cite{crs01}  & eigenvector &  connection\\
      &       &                              &             & \\  \hline
6  & 10  & 9   &    $(0,0,-1,1,0,0)$ &    \\ 
6  & 10  & 9   &    $(0,0,-1,1,0,0)$ &    \\ 
6  & 10  & 9   &    $(0,0,-1,1,0,0)$ &    \\ 
6  & 10  & 13   &    $(0,0,-1,1,0,0)$ &    \\ 
6  & 10  & 13   &    $(-1,0,0,0,0,1)$ &    \\ 
6  & 10  & 14   &    $(0,0,-1,0,1,0)$ &    \\ 
6  & 10  & 16   &    $(-1,0,1,0,0,0)$ &    \\ 
6  & 10  & 18   &   $(-1,0,-1,1,0,1)$  & link 6.31   \\ 
6  & 10  & 18   &   $(1,0,0,0,0,-1)$  & link 6.31   \\ 
6  & 10  & 21   &   $(1,0,0,0,0,-1)$  & link 6.31   \\ 
6  & 10  & 24   &   $(1,0,0,0,0,-1)$  & link 6.31   \\ 
6  & 10  & 29   &   $(0,0,0,1,0,-1)$ & link 6.41    \\ \hline
6  & 9  & 31   &   $(-1,0,-1,1,0,1)$  & link 6.31   \\ 
6  & 9  & 31   &   $(0,-0.6586, -0.2574,0.2574,0.6586,0)$  &    \\ 
6  & 9  & 31   &   $(1,0,0,0,0,-1)$  &    \\ 
6  & 9  & 33   &   $(0,0,-1,1,0,0)$  &    \\ 
6  & 9  & 35   &   $(0,-1,1,0,0,0)$  &    \\ 
6  & 9  & 36   &   $(0,-1,1,0,0,0)$  &  link 6.53  \\ 
6  & 9  & 36   &   $(0,1,-1,0,0,0)$  &  link 6.53  \\ 
6  & 9  & 41   &   $(0,0,0,1,0,-1)$ & link 6.48    \\ 
6  & 9  & 48   &   $(1,-1,1,0,-1,0)$ & link 6.93    \\ 
6  & 9  & 48   &   $(0,0,0,1,0,-1)$ & link 6.49    \\ 
6  & 9  & 49   &   $(0,-1,0,0,1,0)$  &  link 6.78  \\ 
6  & 9  & 49   &   $(-1,0,0,0,0,-1)$  &  link 6.78  \\ \hline
6  & 8  & 53   &   $(-1,1,0,0,0,0)$  &  link 6.78  \\ 
6  & 8  & 53   &   $(0,1,-1,0,0,0)$  &  link 6.78  \\ 
6  & 8  & 53   &   $(0,-1,1,0,0,0)$  &  articulation 5.5 \\ 
6  & 8  & 55   &   $(-1,0,1,0,0,0)$  &  articulation 5.5 \\ 
6  & 8  & 55   &   $(0,-1,1,0,0,0)$  &  articulation 5.5 \\ 
6  & 8  & 61   &   $(0,0,0-2,1,1)$  &  \\ 
6  & 8  & 62   &   $(-1,1,0,0,0,0)$  & link 6.78 \\ 
6  & 8  & 64   &   $(-1,1,0,0,0,0)$  & articulation 5.17  \\ 
6  & 8  & 65   &   $(0,-1,1,0,0,0)$  &            \\ 
6  & 8  & 69   &   1 arbitrary zero  & link 6.93    \\ 
6  & 8  & 71   &   $(1,-1,1,0,-1,0)$ & link 6.93    \\ 
6  & 8  & 73   &   $(0,1,0,0,-1,0)$  &  soldering 5.7      \\ \hline
6  & 7  & 75   &   $(-2,1,0,1,0,0)$  &  articulation 5.18    \\ 
6  & 7  & 78   &   $(0,-1,0,1,0,0)$  &  articulation 5.5    \\ 
6  & 7  & 80   &   $(0,0,-1,0,1,0)$  &  articulation 5.5    \\ 
6  & 7  & 81   &   $(-2,1,0,1,0,0)$  &  articulation 5.18    \\ 
6  & 7  & 82   &   $(0,-1,0,1,0,0)$  &  articulation 5.19    \\ 
6  & 7  & 93   &   $(1,-1,1,0,-1,0)$  &     \\  \hline
\end{tabular}
\caption{\small\em Six node graphs with soft nodes and eigenvalue $4$.}
\label{tab6a}
\end{table}

\subsection{5s}

\begin{table} [H]
\centering
\begin{tabular}{|c|c|c|c|c|}
\hline
nodes & links & classification \cite{crs01}  & eigenvector &  connection \\
      &       &                              &              &  \\  \hline
5  & 10  & 10 &        $(1,-1,0,0,0)$   &          \\ 
5  & 10  & 10 &        $(0,1,-1,0,0)$   &          \\ 
5  & 10  & 10 &        $(0,0,1,-1,0)$   &          \\ 
5  & 10  & 10 &        $(0,0,0,1,-1)$   & link on 5.11 \\ \hline
5  & 9  & 11  &        $ (1,0,-1,0,0)$  & link on 5.12  \\ 
5  & 9  & 11  &        $ (0,0,1,-1,0)$  &          \\ \hline
5  & 8  & 12  &        $(0,1,0,-1,0)$   & link 5.15  \\ \hline
5  & 8  & 13  &        $(1,0,-2,1,0)$   & add 2 soft nodes to 5.3 \\ \hline
5  & 7  & 15  &        $(1,1,-3,1,0)$   & add soft node to 4 star \\ \hline
5  & 7  & 15  &        $(0,0,-1,1,0)$   & add 3 soft nodes to 5.1 \\ \hline
\end{tabular}
\caption{\small\em Five node graphs with soft nodes and eigenvalue $5$.}
\label{tab7}
\end{table}

\begin{table} [H]
\centering
\begin{tabular}{|c|c|c|c|c|}
\hline
nodes & links & classification \cite{crs01}  & eigenvector &  connection \\
      &       &                              &              &  \\  \hline
6  & 13 & 3    &   $(1,0,0,0,0,-1)$  &          \\ \hline
6  & 12 & 5    &   $(-1,1,0,0,0,0)$ & link 6.12 \\ 
6  & 12 & 5    &   $(1,1,0,0,0,-2)$ &  link 6.14        \\ \hline
6  & 12 & 8    &   $(0,0,0,0,1,-1)$  &  link 6.12        \\ 
6  & 12 & 8    &   $(1,0,0,-1,0,0)$  &          \\ \hline
6  & 11 & 10   &   $(1,-1,0,0,0,0)$  &  link 6.19  \\ 
6  & 11 & 10   &   $(-1,1,-1,1,0,0)$ & link 6.19 \\ 
6  & 11 & 10   &   $(1,-1,-1,1,0,0)$ &  \\ 
6  & 11 & 11   &   $(1,1,0,0,0,-2)$  &  link 6.14    \\ 
6  & 11 & 12   &   $(1,-1,0,0,0,0)$  &  link 6.14  \\ 
6  & 11 & 14   &   $(1,1,0,0,0,-2)$  &  link 6.29     \\ 
6  & 11 & 14   &   $(1,-1,0,0,0,0)$  &          \\ 
6  & 11 & 17   &   $(-1,0,0,1,-1,1)$  & inserting a matching between two 5.2 \\ 
6  & 11 & 17   &   $(0,0,0,1,-1,0)$  & add 3 soft nodes to 5.1 \\ \hline
6  & 10 & 19   &   $(-1,1,1,-1,0,0)$ & link on 6.30 \\      
6  & 10 & 19   &   $(0,1,-1,0,0,0)$  & link on 6.19 \\
6  & 10 & 20   &   $(1,0,-2,1,0,0)$  & link on 6.33  \\ 
6  & 10 & 20   &   $(0,0,-1,1,0,0)$  &  link on 6.35\\ 
6  & 10 & 23   &   $(0,1,0,0,-1,0)$  & link on 6.32 \\ 
6  & 10 & 23   &   $(-1,0,-1,1,0,1)$ & link on 6.30 \\ 
6  & 10 & 26   &   $(0,0,0,1,0,-1)$  & link on 6.32         \\ 
6  & 10 & 29   &   $(1,-2,1,0,0,0)$  & link on 6.32 \\ 
6  & 10 & 30   &   $(1,-1,0,1,-1,0)$ & link on 6.70         \\ 
6  & 10 & 32   &   $(0,1,0,0,0,-1)$  & link on 6.32 \\ \hline
6  & 9 & 33   &   $(1,-3,1,1,0,0)$  & link on 6.56\\
6  & 9 & 34   &   $(0,1,0,-1,0,0)$  & link on 6.45 \\
6  & 9 & 35   &   $(1,0,0,-1,0,0)$  & link on 6.45 \\ 
6  & 9 & 38   &   $(1,-1,1,-1,0,0)$ & articulation on 5.13 \\ 
6  & 9 & 39   &   $(1,-2,0,1,0,0)$  & articulation on 5.13  \\ 
6  & 9 & 44   &   $(-1,0,1,-1,0,1)$  &  link on 6.70 \\ 
6  & 9 & 45   &   $(0,0,1,-1,0,0)$   &  link on 6.57 \\ \hline
6  & 9 & 51   &   $(0,0,1,-1,1,-1)$   &  link on 6.70 \\ \hline
6  & 8 & 56   &   $(1,1,-3,1,0,0)$   &  articulation on 5.15 \\ 
6  & 8 & 57   &   $(0,-1,0,1,0,0)$   &  articulation on 5.15 \\ 
6  & 8 & 70   &   $(-1,0,1,-1,0,1)$  &  insert matching on cycle 3 \\ \hline
\end{tabular}
\caption{\small\em Six node graphs with soft nodes and eigenvalue $5$.}
\label{tab7a}
\end{table}

\subsection{6s}

\begin{table} [H]
\centering
\begin{tabular}{|c|c|c|c|c|}
\hline
nodes & links & classification \cite{crs01}  & eigenvector &  connection \\
      &       &                              &              &  \\  \hline
6  & 13 & 1    &   $(1,-1,0,0,0,0)$  &          \\ 
6  & 13 & 1    &   $(1,0,-1,0,0,0)$  &          \\ 
6  & 13 & 1    &   $(1,0,0,-1,0,0)$  &          \\ 
6  & 13 & 1    &   $(1,1,-2,0,0,0)$  & link 6.3         \\ 
6  & 13 & 2    &   $(1,-1,0,0,0,0)$  &          \\ 
6  & 13 & 2    &   $(1,0,-1,0,0,0)$  &          \\ 
6  & 13 & 2    &   $(1,0,0,-1,0,0)$  &          \\ 
6  & 13 & 3    &   $(1,-2,0,1,0,0)$  & link 6.4         \\ \hline
6  & 12 & 3    &   $(0,1,0,0,-1,0)$ & link 6.4 \\ 
6  & 12 & 4    &   $(1,-2,0,1,0,0)$  &          \\ 
6  & 12 & 4    &   $(0,1,0,0,-1,0)$ & link 6.6 \\ 
6  & 12 & 5    &   $(0,0,1,0,-1,0)$ & link 6.7 \\ 
6  & 12 & 6    &   $(1,1,0,0,0,-2)$ &  link 6.16        \\ 
6  & 12 & 6    &   $(1,0,-1,0,0,0)$ &  link 6.7        \\ 
6  & 12 & 7   &   $(0,1,0,0,-1,0)$  &  add four soft nodes to 5.1        \\ 
6  & 12 & 8   &   1 arbitrarily placed 0  &          \\ \hline
6  & 11 & 9   &   $(1,1,0,-1,1,0)$ &  add two soft nodes to 5.7\\ 
6  & 11 & 11   &   $(1,1,1,-4,1,0)$  &  add a soft node to 5.28 \\ 
6  & 11 & 13   &   1 arbitrarily placed 0  &          \\ 
6  & 11 & 16   &   $(1,0,1,0,-2,0)$  & add 3 soft nodes to 5.2 \\ \hline
6  & 9 & 37   &   $(1,0,0,-1,0,0)$ & add 4 soft node to 5.1  \\ \hline
\end{tabular}
\caption{\small\em Six node graphs with soft nodes and eigenvalue $6$.}
\label{tab8a}
\end{table}

\end{document}